\documentclass[letter,11pt]{article}
\usepackage{pdfsync}
\usepackage[utf8]{inputenc}
\usepackage{amsmath,amssymb,amsfonts,amsbsy,setspace, amsthm,}
\usepackage{mathrsfs}
\usepackage{verbatim}
\usepackage{enumerate}
\usepackage{eucal}

\usepackage{anysize}
\marginsize{4cm}{4cm}{3cm}{3cm}
\usepackage{cite}
\usepackage{hyperref}
\usepackage{graphicx}
\usepackage{xcolor}
\newtheorem{theorem}{Theorem}
\newtheorem{corollary}[theorem]{Corollary}
\newtheorem{lemma}[theorem]{Lemma}
\newtheorem{proposition}[theorem]{Proposition}
\newtheorem{remark}[theorem]{Remark}
\newtheorem{definition}[theorem]{Definition}
\newtheorem{example}[theorem]{Example}

\DeclareMathOperator{\argmin}{argmin}

\DeclareMathOperator{\X}{X}
\DeclareMathOperator{\Y}{Y}

\def\<#1,#2>{\left<#1,#2\right>}
\let\bar\overline

\DeclareMathOperator{\MM}{MM}
\renewcommand\phi{\varphi}

\renewcommand\epsilon{\varepsilon}

\def\PP{{\cal P}}

\title{{Transport type metrics on the space of probability measures involving singular base measures}\footnote{B.P. is pleased to acknowledge the support of National Sciences and Engineering Research Council of Canada Discovery Grant number  04658-2018.  He is also grateful for the kind hospitality at the Institut de Mathématiques d'Orsay, Université Paris-Sud during his stay in November of 2019 as a missionaire scientifique invité, when this work was partially completed. L.N. would like to acknowledge the support from the project MAGA ANR-16-CE40-0014 (2016-2020) and the one from the CNRS PEPS JCJC (2021).   They are also indebted to two anonymous referees for their many insightful comments on an earlier draft of this work.}}

\author{
Luca Nenna\thanks{Université Paris-Saclay, CNRS, Laboratoire de mathématiques d’Orsay, 91405, Orsay, France.  \texttt{luca.nenna@universite-paris-saclay.fr}}  %$\,^{*}$
\and 
Brendan Pass\thanks{Department of Mathematical and Statistical Sciences, 632 CAB, University of Alberta, Edmonton, Alberta, Canada, T6G 2G1  \texttt{pass@ualberta.ca}} 
}

\begin{document}

\maketitle
 
 \begin{abstract}
  We develop the theory of a metric, which we call the $\nu$-based Wasserstein metric and denote by $W_\nu$, on the set of probability measures $\mathcal P(X)$ on a domain $X \subseteq \mathbb{R}^m$.  This metric is based on a slight refinement of the notion of generalized geodesics with respect to a base measure $\nu$ and is relevant in particular for the case when $\nu$ is singular with respect to $m$-dimensional Lebesgue measure; it is also closely related to the concept of linearized optimal transport.    The $\nu$-based Wasserstein metric is defined in terms of an iterated variational problem involving optimal transport to $\nu$; we also characterize it in terms of integrations of classical Wasserstein metric between the conditional probabilities when measures are disintegrated with respect to optimal transport to $\nu$, and through limits of certain multi-marginal optimal transport problems.
 We also introduce a class of metrics which are dual in a certain sense to $W_\nu$, defined relative to a fixed based measure $\mu$, on the set of measures which are absolutely continuous with respect to a second fixed based measure $\sigma$.
  
 As we vary the base measure $\nu$, the $\nu$-based Wasserstein metric interpolates between the usual quadratic Wasserstein metric (obtained when $\nu$ is a Dirac mass) and a metric associated with the uniquely defined generalized geodesics obtained when $\nu$ is sufficiently regular (eg, absolutely continuous with respect to Lebesgue).  When $\nu$ concentrates on a lower dimensional submanifold of $\mathbb{R}^m$, we prove that the variational problem in the definition of the $\nu$-based Wasserstein metric has a unique solution.  We also establish geodesic convexity of the usual class of functionals,   and of the set of source measures $\mu$ such that optimal transport between $\mu$ and $\nu$ satisfies a strengthening of the generalized nestedness condition introduced in \cite{McCannPass20}.   
We finally introduce a slight variant of the dual metric mentioned above in order to prove convergence of an iterative scheme to solve a variational problem arising in game theory.
% We also present two applications of the ideas introduced here.  First, our dual metric (in fact, a slight variant of it) is used to prove convergence of an iterative scheme to solve a variational problem arising in game theory.  We also use the multi-marginal formulation to characterize solutions to the multi-marginal problem by an ordinary differential equation, yielding a new numerical method for it.
 \end{abstract}

 \tableofcontents

 \section{Introduction}
 Given two probability measures $\mu_0$ and $\mu_1$ on a convex, bounded domain $X \subseteq \mathbb{R}^m$, the Wasserstein distance between them is defined as the infimal value in the Monge-Kantorovich optimal transport problem; that is, 
 \begin{equation}\label{eqn: ot with quadratic cost}
 W_2(\mu_0,\mu_1) :=\sqrt{\inf_{\gamma \in \Pi(\mu_0,\mu_1)}\int_{X \times X}|x_0-x_1|^2d\pi(x_0,x_1)}
 \end{equation}
 where the infimum is over the set $\Pi(\mu_0,\mu_1)$ of joint measures $\gamma$ on $X \times X$ whose marginals are $\mu_0$ and $\mu_1$.

 Among the many important properties of the Wasserstein distance (reviewed in \cite{santambook}\cite{Villani-TOT2003} and \cite{Villani-OptimalTransport-09} for example) is the fact that it is a metric on the set $\mathcal P(X)$ of probability measures on $X$.  In turn, the geodesics induced by this metric, known as \emph{displacement interpolants}  and introduced in \cite{mccann1997convexity}, play key roles in many problems, both theoretical and applied.    Variants of displacement interpolants, known as \emph{generalized geodesics} (introduced in \cite{AmbrosioGradientFlows2008}, see Definition \ref{def: generalized geodesics} below), are a natural and important tool in the analysis of problems involving a fixed base measure $\nu \in \mathcal P(X)$.  For example, a variety of problems in game theory, economics and urban planning consist in minimizing functionals on $ \mathcal P(X)$ involving optimal transport to a fixed $\nu$. In particular, we mention the optimal transport based formulation of Cournot-Nash equilibria in game theory of Blanchet-Carlier, in which $\nu$ parameterizes a population of players and one searches for their equilibrium distribution of strategies, parameterized by $\mu \in \mathcal P(X)$ \cite{blanchet2014remarks}\cite{abgcptrl}\cite{abgcmor}.  When $\nu$ is absolutely continuous with respect to the Lebesgue measure $\mathcal{L}^m$ on $X$, these interpolants are uniquely defined, and are in fact geodesics for a metric under which $\mathcal P(X)$   is isomorphic to a subset of the Hilbert space $L^2(X,\mathbb{R}^n)$; that is, under this isomorphism, generalized geodesics are mapped to line segments.  
 
  On the other hand, it is natural in certain problems to consider \emph{singular} base measures\footnote{For instance, in game theory problems derived from spatial economics, $\nu$ may represent a population of players, parametrized by their location $y \in \mathbb{R}^2$ \cite{abgcmor}; it is often the case that the population is essentially concentrated along a one dimensional subset, such as a major highway or railroad.}, and our goal here is to initiate the development of a framework to study these. We study a metric on an appropriate subset of $\mathcal P(X)$ for each choice of base measure $\nu$, which we call the \textit{$\nu$-based Wasserstein metric} (see Definition \ref{def: nu based definition}); essentially, this metric arises from minimizing the average squared distance among all couplings of $\mu_0$ and $\mu_1$ corresponding to generalized geodesics.  It is also very closely related to the concept of linear optimal transport, introduced in \cite{WangSlepcevBasuOzolekRohde13}\footnote{In fact, in the case we are most interested in, when there exists a unique optimal transport between $\nu$ and each of the measures to be compared, the $\nu$-based Wasserstein metric coincides with the metric derived from solving the linear optimal transport problem, denoted by $d_{LOT,\nu}$ in \cite{WangSlepcevBasuOzolekRohde13}.  They differ slightly for more general measures; in this case, the $\nu$-based Wasserstein metric is in fact only a semi-metric.}. Linear optimal transport  was mainly used in \cite{WangSlepcevBasuOzolekRohde13} for image comparison problems, and the base measure $\nu$ was often taken to be discrete.  In that work, and, to the best of our knowledge, in the considerable literature on linear optimal transport following from it, the theoretical properties of the metric, especially for base measures supported on lower dimensional submanifolds, to which we pay special attention below, have largely been left undeveloped. 
  
    Geodesics with respect to the $\nu$-based Wasserstein metric will always be generalized geodesics with respect to $\nu$, but, for different structures of $\nu$, very different geometry is induced.  In particular, when $\nu$ is absolutely continuous, we obtain the Hilbert space geometry discussed above, whereas when $\nu$ is a Dirac mass, we obtain the Wasserstein metric, independently of the $y$ parameterizing $\nu = \delta_y$ within this class. We pay special attention to the cases in between, in particular when $\nu$ concentrates on a lower dimensional submanifold of $\mathbb{R}^m$, in which case the problem has a natural interplay with the unequal dimensional optimal transport problem explored in \cite{ChiapporiMcCannPass15p}, \cite{PassM2one}, \cite{McCannPass20} and \cite{NennaPass1}. In the particular case when $\nu$ concentrates on a line segment, we show that our metric coincides with the layerwise-Wasserstein metric (see \eqref{eqn: layerwsie Wasserstein def}) introduced in \cite{KimPassSchneider19} to analyze anisotropic data such as plants' root shapes. 
 
 We establish three equivalent characterizations of the $\nu$-based Wasserstein metric, roughly speaking:
 \begin{enumerate}
\item as an optimal transport problem restricted to couplings which are correlated along $\nu$ (we take this as the definition);
 \item by optimally coupling conditional probabilities of $\mu_0$ and $\mu_1$ after disintegrating with respect to optimal transport to $\nu$;
  \item as limits of multi-marginal optimal transport between $\mu_0$, $\mu_1$ and $\nu$.
  \end{enumerate}
    In many cases of interest, we establish uniqueness of the corresponding geodesics (although uniqueness of generalized geodesics for regular base measures was established in \cite{AmbrosioGradientFlows2008}, they are generally not unique when the base measure is singular).   We also study geodesic convexity of several functionals which play a key role in optimal transport research;  many of these functionals were originally introduced by McCann \cite{mccann1997convexity}, who established their displacement convexity. Using the standard terminology (see Section \ref{sect: convexity} for the definition of these terms) convexity  along $W_\nu$ geodesics  of potential energies, interaction energies and the Wasserstein distance to $\nu$ follow immediately from known results ( namely, their convexity along any generalized geodesic, established in \cite{AmbrosioGradientFlows2008}), whereas convexity along $W_\nu$ geodesics of the internal energy, under certain conditions, requires a new proof (convexity of the internal energy is known to hold along some, but not all, generalized geodesics \cite{AmbrosioGradientFlows2008}).   We note that this applies in particular to the layerwise-Wasserstein distance, yielding a far reaching improvement to Corollary 4.2 in \cite{KimPassSchneider19}.  We also show that when $\nu$ concentrates on a lower dimensional submanifold,  the set of measures $\mu$ for which the model $(|x-y|^2,\mu,\nu)$ satisfies a strengthening of the generalized nested condition ( see Definition \ref{def: generalized nestedness}) is geodesically convex (we recall that nestedness, introduced in \cite{PassM2one} and its higher dimensional generalization from \cite{McCannPass20}, are important properties of unequal dimensional  optimal transport problems; when present, they  greatly simplify analysis of these problems). 
 
 We also introduce a class of metrics which is in a certain sense dual to the $\nu$-based Wasserstein metric, relevant in the case when the measure $\mu$ on $X \subset \mathbb{R}^m$ is fixed, and one would like to interpolate between measures on a fixed, lower dimensional submanifold $Y$ (see the precise definition \eqref{eqn: dual metric}). This is often the case in a variety of applications.  The seemingly natural choice, generalized geodesics with base $\mu$, which essentially interpolates between Kantorovich potentials,  namely the optimal dual variable of the optimal transport problem, on the $X$ side, does not generally result in interpolants supported on $Y$. Here, we instead compare and interpolate between Kantorovich potentials on the $Y$ side in order to compare and interpolate between measures on $Y$.  This is in a certain sense complementary, or dual, to our original metric on $\mathcal P(X)$ (which involves comparing potentials on $Y$, in order to compare measures on $X$ -- with an additional embedded optimization problem, since potentials on $Y$ do not uniquely determine measures supported on $X$).
 We, finally, show how the ideas introduced here can be applied to prove convergence of computational methods to find equilibria in some game theoretic models.
 In particular we identify conditions under which equilibria  are fixed points of a contractive mapping, implying uniqueness of the equilibrium (although this is easily deduced by other methods as well) and, perhaps more importantly, that it can be computed by iterating the mapping.
 %We also present two applications of the ideas introduced here. 
 % First, we identify conditions under which equilibria in certain game theoretic models are fixed points of a contractive mapping, implying uniqueness of the equilibrium (although this is easily deduced by other methods as well) and, perhaps more importantly, that it can be computed by iterating the mapping.  
 This iteration had already been introduced as a method of computation by Blanchet-Carlier when $m=1$, but without a proof of convergence, and in higher dimensions, with a proof of convergence but for simpler interaction terms \cite{blanchet2014remarks}.  Here, we prove that the relevant mapping is a contraction with respect to a variant of the dual metric described above\footnote{In fact, the actual metric used in the proof differs slightly from the dual metric in general, although they coincide under certain conditions.}.  %Roughly speaking, for fixed $\mu$, interpolates between Katnorovich potentials on the target, rather than the source, as is the case for the usual generalized geodesics. %%Seen another way, it uses the Kantorovich potential  on the low dimensional side to interpolate between measures on the same side, whereas our original metric uses Kantorovich potentials on the low dimensional side to interpolate between measures on the high dimensional side (in part, since Kantorovich potentials on the low dimensional side do not entirely determine the higher dimensional source measures).
 
% Second, we use the characterization of our metric in terms of limits of multi-marginal optimal transport to establish an ordinary differential equation for solutions to the multi-marginal problem and analyze its properties.  We note that, even for absolutely continuous base measures, this connection between generalized geodesics and multi-marginal problems does not seem to have been observed before.  Since the initial condition in our differential equation arises from solving two marginal optimal transport problems (which are generally much easier to solve than  multi-marginal problem ), this ODE yields a numerical method to solve the multi-marginal problem which has certain advantages over existing methods.
 
 The manuscript is organized as follows.  In the next section we introduce the $\nu$-based Wasserstein metric, establish several basic properties of it, and also introduce our class of dual metrics. In Section 3, we recall relevant facts about unequal dimensional optimal transport, and prove a new lemma on the structure of optimal plans which will be crucial in subsequent sections.  In the fourth section, we identify conditions under which the variational problem arising in the definition of the $\nu$-based Wasserstein metric has a unique solution, and establish a result on the structure of geodesics for the $\nu$-based Wasserstein metric.  We use this structure in Section 5 to establish geodesic convexity results.  In the sixth section, we identify conditions under which certain game theoretic equilibria can be characterized by fixed points of a contractive mapping with respect to an appropriate metric.%, while the last section is reserved for the development of an ordinary differential equation description of multi-marginal optimal transport.
\section{Definition and basic properties}
Throughout this paper, for a given set $S$, a semi-metric will be a function $d:S \times S \rightarrow [0,\infty)$ satisfying:
	\begin{itemize}
		\item $d(x,y)=0$ if and only if $x=y$ (separation).
		\item $d(x,y) =d(y,x)$ for all $x,y \in S$ (symmetry).
	\end{itemize}

A semi-metric $d$ will be called a metric if it also satisfies:
\begin{itemize}
	\item $d(x,y) \leq d(x,z) +d(y,z)$ for all $x,y,z \in S$ (triangle inequality).
\end{itemize}
We will often refer to $d(x,y)$ as the distance between $x$ and $y$.
\subsection{Background on optimal transport}
In what follows, $\nu \in \mathcal P(X)$ will be a fixed reference measure on a  convex, bounded\footnote{We have chosen to work on a bounded set $X$ here mostly out of technical convenience and to keep the presentation simple.  We expect that most results can be extended to unbounded domains under appropriate hypotheses (for instance, decay conditions on the measures).  } domain $X \subseteq \mathbb{R}^m$.  Given measures $\mu_0,\mu_1 \in \mathcal P(X)$, we denote by $\Pi(\nu,\mu_i)$ the set of probability measures on $X \times X$ whose marginals are $\nu$ and $\mu_i$, for $i=0,1$ and by $\Pi(\nu,\mu_0, \mu_1)$ the set of probability measures on $X \times X \times X$ with $\nu,\mu_0$ and $\mu_1$ as marginals.  For a measure $\gamma \in \mathcal P(X\times X\times X)$, we will denote its first, second and third marginals by $\gamma_y, \gamma_{x_0}$ and $\gamma_{x_1}$, respectively.  Similarly, we will denote by $\gamma_{yx_0}$, $\gamma_{yx_1}$ and $\gamma_{x_0x_1}$ its projections onto the appropriate  product $X \times X$; for example, $\gamma_{yx_0}$ is the push forward of $\gamma$ under the mapping $(y,x_0,x_1) \mapsto (y,x_0)$.

The optimal transport problem between $\mu_i$ and $\nu$, for a general continuous cost function $c:\bar X \times \bar X \rightarrow \mathbb{R}$ is to minimize over $\pi \in \Pi(\mu_i,\nu)$
	\begin{equation}\label{eqn: ot with general cost}
		\int_{X\times X}c(x_i,y)d\pi(x_i,y).
	\end{equation}
We will most often be interested in the quadratic cost function, $c(x_i,y) = |x_i-y|^2$,  where $|\cdot|$ denotes here the standard euclidean norm. Notice that in the quadratic case, \eqref{eqn: ot with general cost} is the problem arising in \eqref{eqn: ot with quadratic cost}.  In addition, \eqref{eqn: ot with general cost} has a well known dual problem; the minimum in \eqref{eqn: ot with general cost} is equal to the maximum of
\begin{equation}\label{eqn: ot dual}
	\int_X u(x)d\mu_i(x) +\int_Xv(y)d\nu(y) 
\end{equation}
among pairs of functions $(u,v) \in L^1(\mu_i) \times L^1(\nu)$ with $u(x) + v(y) \leq c(x,y)$.  It is well known that maximizers to \eqref{eqn: ot dual} exist and they may be taken to be $c$-conjugate; that is,
$$
u(x) =v^c(x):=\min_{y \in X}[c(x,y) - v(y)],\text{   }v(y) =u^c(y):=\min_{x \in X}[c(x,y) - u(x)].
$$ 
We let $\Pi_{opt}(\nu, \mu_i)$ be the set of optimal couplings between $\nu$ and $\mu_i$ with respect to optimal transport for the quadratic cost function; that is:
\begin{equation}\label{eqn: optimal matching}
\Pi_{opt}(\nu, \mu_i):=\argmin_{\pi \in \Pi(\nu,\mu_i)} \int_{X\times X}|x_i-y|^2d\pi(x_i,y).
\end{equation}

 \begin{definition}
 A \emph{Wasserstein geodesic} between $\mu_0$ and $\mu_1$ (also known as a \emph{displacement interpolant}, using the terminology originally introduced by McCann \cite{mccann1997convexity}) is a curve $\mu_t$ in $\mathcal P(X)$ of the form $\mu_t=((1-t)e_0 +te_1)_\#\pi$ for some $\pi \in \Pi_{opt}(\mu_0, \mu_1)$, where for $i=0,1$, $e_i:X \times X \rightarrow X$ is defined by $e_i(x_0,x_1) =x_i$.  
 
 \end{definition}

%A \emph{Wasserstein geodesic} between $\mu_0$ and $\mu_1$ (also known as a \emph{displacement interpolant}, using the terminology originally introduced by McCann \cite{mccann1997convexity}) is a curve $\mu_t$ in $\mathcal P(X)$ of the form $\mu_t=((1-t)e_0 +te_1)_\#\pi$ for some $\pi \in \Pi_{opt}(\mu_0, \mu_1)$, where for $i=0,1$, $e_i:X \times X \rightarrow X$ is defined by $e_i(x_0,x_1) =x_i$.  

Recall that, in general, a curve $z_t$ parametrized by $t \in [0,1]$ in a metric space $(Z,d)$ is a \textit{minimizing geodesic} if  $d(z_s,z_t) = |s-t|d(z_0,z_1)$ for all $s,t \in [0,1]$.  It is well known that the Wasserstein geodesic defined above is in fact a geodesic in this sense for the Wasserstein metric.%It is well known that for any $\mu_0,\mu_1 \in \mathcal P(X)$, where $\mu_0$ is absolutely continuous with respect to Lebesgue measure, there is a unique Wasserstein geodesics joining them, also know as the displacement interpolant and given by $\mu_t = \Big((1-t)I+t\nabla \phi \Big)_\#\mu_0$, where $\nabla \phi$ is the Brenier map between $\mu_0$ and $\mu_1$.

We now give the definition of generalized geodesics from \cite{AmbrosioGradientFlows2008}.
\begin{definition}
\label{def: generalized geodesics}
A generalized geodesic with base measure $\nu$ from $\mu_0$ to $\mu_1$ is a curve $\mu_t$ in  $\mathcal P(X)$ of the form $\mu_t=((1-t)e_0 +te_1)_\#\gamma$ for some $\gamma \in \mathcal{P}(X \times X \times X)$ where $\gamma_{yx_i} \in  \Pi_{opt}(\nu, \mu_i)$ for $i=0,1$.
\end{definition}

We note that we are especially interested here in the case where $\nu$ is singular with respect to Lebesgue measure.  Even in this case, $\Pi_{opt}(\nu, \mu_i)$ will very often consist of a single probability measure; in fact, by Brenier's theorem this is the case as soon as $\mu_i$ is absolutely continuous with respect to Lebesgue measure \cite{brenier1991polar}. It will turn out that our definition below is not a metric on all of $\mathcal P(X)$, but is when restricted to the set $\mathcal P_\nu^u(X)$ of probability measures on $X$ for which the solution to the optimal transport problem to $\nu$ is unique; that is, the set such that $\Pi_{opt}(\nu,\mu)$ is a singleton.

Before closing this section, we recall the well known notion of disintegration of measures.  Given $\pi_i \in \Pi(\nu,\mu_i)$, we disintegrate with respect to $\nu$. There exists a measurable collection $\mu_i^y$ of probability measures on $X$, indexed by $y$, such that 
	$$
	\pi_i(x,y) = \nu(y) \otimes \mu_i^y(x);
	$$
	that is, for every bounded continuous function $g$ on $X \times X$, we have
	$$
	\int_{X \times X}g(x,y)d\pi_i(x,y) = \int_X\bigg[\int_Xg(x,y)d\mu_i^y(x)\bigg ]d\nu(y).
	$$
	We will sometimes call $\mu_i^y$ the \emph{conditional probability} of $\pi_i$ given $y$.

\subsection{The $\nu$-based Wasserstein metric}
We now define our metric with base point $\nu$ as follows.
%\Todo{Recall geodesics, generalized geodesics, general optimal transport and dual problem, etc. here?}
%[BP: WE SHOULD NAME THIS METRIC SOMEHOW ]

\begin{definition}\label{def: nu based definition}
	Let $\nu \in \mathcal P(X)$.  For $\mu_0,\mu_1 \in \mathcal P(X)$, we define the \textit{$\nu$-based Wasserstein metric} as
	\begin{equation}\label{eqn: metric definition}
	W_\nu(\mu_1,\mu_0):=\sqrt{\inf_{\gamma \in\Gamma }\int_{X\times X\times X}|x_0-x_1|^2d\gamma(y,x_0,x_1)},
	\end{equation}
	where $\Gamma:=\{ \gamma\in\mathcal P(X \times X \times X)\;|\;\gamma_{yx_i}\in \Pi_{opt}(\mu_i,\nu), i=0,1\}$.
\end{definition}
\begin{remark}
	As we will see below, the name  $\nu$-based Wasserstein metric is a slight abuse of terminology, since  $W_\nu$ is only a semi-metric on $\mathcal{P}(X)$ in general, although it is a metric on an appropriate subset.  
\end{remark}
Note that the glueing lemma (see \cite{santambook}[Lemma 5.5], for example) implies the existence of a $\gamma \in \Pi(\nu,\mu_0,\mu_1)$ such that $\gamma_{yx_i}  \in \Pi_{opt}(\nu, \mu_i)$ for $i=0,1$; therefore, $W_\nu$ is well defined. Standard arguments imply the existence of a minimizing $\gamma$ in \eqref{eqn: metric definition}.
\begin{remark}
	This definition is closely related to the concept of linear optimal transport, introduced in \cite{WangSlepcevBasuOzolekRohde13}.  The difference is that in linear optimal transport, a \emph{fixed} optimal transport $\pi_{\mu_i} \in \Pi_{opt}(\nu,\mu_i)$ is selected for each $\mu_i$ (see equation (3) in \cite{WangSlepcevBasuOzolekRohde13}), whereas in the definition of $W_\nu$ one minimizes over the entire set $\Pi_{opt}(\nu,\mu_i)$.  For $\mu_i \in \mathcal{P}^u_\nu(X)$, the two concepts clearly coincide, and it is on this set that $W_\nu$ yields a metric (see Lemma \ref{lem: metric for unique ot} below).  Outside of this set, $W_\nu$ still yields a semi-metric, whereas linear optimal transport might be better described as defining a metric on the selected $\pi_{\mu_i} \in \Pi_{opt}(\nu,\mu_i)$, since it is dependent on these choices; this idea will be developed briefly in Remark \ref{rem: metric on couplings} below, but first let us rigorously establish that $W_\nu$ is indeed a semi-metric.
\end{remark}

The proof of the following Lemma is very similar to the proof that the classical Wasserstein metric is in fact a metric (see for example,  \cite{santambook}[Proposition 5.1]). Recall that a semi-metric satisfies the symmetry and identity of indiscernibles axioms in the definition of a metric, but does not satisfy the triangle inequality.
\begin{lemma}\label{lem: metric for unique ot}
	$W_\nu$ is a semi-metric on $\mathcal P(X)$. It is a metric on $\mathcal P_\nu^u(X)$.
\end{lemma}
%[BP :UPDATE -- THIS ISN'T A METRIC, UNLESS WE RESTRICT THE MEASURES SOMEHOW; FOR INSTANCE, WE COULD REQUIRE   ABSOLUTE CONTINUITY. ]
\begin{proof}
	It is immediate that $W_\nu(\mu_0,\mu_1) \geq 0$, with equality if and only if $\mu_0 =\mu_1$, and that $W_\nu(\mu_0,\mu_1) =W_\nu(\mu_1,\mu_0)$.  Therefore, $W_\nu$ is a semi-metric.  
	
	It remains to show that $W_\nu$ is a metric on $\mathcal{P}_\nu^u(X)$; we must only verify the triangle inequality. Let $\mu_0,\mu_1$ and $\mu_2$ belong to $\mathcal{P}_\nu^u(X)$.   Let $\gamma_1 \in \Pi(\nu,\mu_0,\mu_1)$ and $\gamma_2\in \Pi(\nu,\mu_0,\mu_2)$ be optimal couplings in \eqref{eqn: metric definition}; that is $W_\nu(\mu_0,\mu_1)=\int_{X\times X\times X}|x_0-x_1|^2d\gamma_1(y,x_0,x_1)$ and $W_\nu(\mu_0,\mu_2)=\int_{X\times X\times X}|x_0-x_2|^2d\gamma_2(y,x_0,x_2)$.  Now note that both $(\gamma_1)_{yx_0}$ and $(\gamma_2)_{yx_0}$ are both optimal transports between $\nu$ and $\mu_0$; by the uniqueness assumption, we therefore have $(\gamma_1)_{yx_0}=(\gamma_2)_{yx_0}$.  The glueing lemma (the version in \cite{Villani-TOT2003}, Lemma 7.6, is sufficiently general) then implies the existence of a measure $\gamma \in \mathcal P(X \times X \times X \times X)$ such that $\gamma_{yx_0x_1} =\gamma_1$ and $\gamma_{yx_0x_2} =\gamma_2$.  We note that $\gamma_{yx_1x_2}$ satisfies $(\gamma_{yx_1x_2})_{yx_1} =\gamma_{yx_1} =(\gamma_1)_{yx_1} \in \Pi_{opt}(\nu,\mu_1)$ and $(\gamma_{yx_1x_2})_{yx_2}=\gamma_{yx_2} =(\gamma_2)_{yx_2} \in \Pi_{opt}(\nu,\mu_2)$.  Therefore, we have
	\begin{eqnarray*}
	W_\nu(\mu_1,\mu_2)&\leq &\sqrt{\int_{X\times X\times X}|x_1-x_2|^2d\gamma_{yx_1x_2}(y,x_1,x_2)}\\
	&=&\sqrt{\int_{X\times X\times X \times X}|x_1-x_2|^2d\gamma(y,x_0,x_1,x_2)}\\
	&=&||x_1-x_2||_{L^2(\gamma)}\\
	&\leq &||x_1-x_0||_{L^2(\gamma)} +||x_2-x_0||_{L^2(\gamma)}\\
	&=&\sqrt{\int_{X\times X\times X \times X}|x_1-x_0|^2d\gamma(y,x_0,x_1,x_2)} +\sqrt{\int_{X\times X\times X \times X}|x_0-x_2|^2d\gamma(y,x_0,x_1,x_2)}\\
	&=&\sqrt{\int_{X\times X\times X}|x_1-x_0|^2d\gamma_{1}(y,x_0,x_1)}+\sqrt{\int_{X\times X\times X}|x_0-x_2|^2d\gamma_{2}(y,x_0,x_2)}\\
	&=&	W_\nu(\mu_0,\mu_1)+	W_\nu(\mu_0,\mu_2)
	\end{eqnarray*}
\end{proof}
%\begin{proof}
%Symmetry $W_\nu(\mu_1,\mu_0) =W_\nu(\mu_0,\mu_1)$ is obvious.  Furthermore, it is clear that $W_\nu(\mu_1,\mu_0) \geq 0$, with equality if and only if 
%\end{proof}
The following example confirms that the triangle inequality can fail if we do not restrict to $ \mathcal P^u_\nu(X)$.
\begin{example}(Failure of the triangle inequality outside of $\mathcal{P}_\nu^u(X)$)
	Let $X \subseteq \mathbb{R}^2$, and take $\nu = \frac{1}{2}[\delta_{(1,0)}+\delta_{(-1,0)}]$, $\mu_0 = \frac{1}{2}[\delta_{(0,1)}+\delta_{(0,-1)}]$, $\mu_1 = \frac{1}{2}[\delta_{(\epsilon,1)}+\delta_{(-\epsilon,-1)}]$, $\mu_2 = \frac{1}{2}[\delta_{(-\epsilon,1)}+\delta_{\epsilon,-1)}]$, for some  $\epsilon >0$.  Then any measure $\pi \in \Pi(\nu,\mu_0)$ is optimal between $\nu$ and $\mu_0$.  The only optimal transport plan between $\nu$ and $\mu_1$, on the other hand, maps $( 1,0)$ to $( \epsilon,1)$ and $( -1,0)$ to $(-\epsilon,-1)$.  Similarly, the only optimal transport plan between $\nu$ and $\mu_2$, maps $( 1,0)$ to $( \epsilon,-1)$ and $( -1,0)$ to $( -\epsilon,1)$.  We therefore compute that
	
	$$
	W_\nu(\mu_0,\mu_1) =\epsilon=W_\nu(\mu_0,\mu_2)
	$$
	but
	$$
		W_\nu(\mu_1,\mu_2) =2.
	$$ 
	Therefore, $W_\nu(\mu_1,\mu_2) >	W_\nu(\mu_0,\mu_1)+W_\nu(\mu_0,\mu_2)$ for $\epsilon <1$, and so the triangle inequality fails.
\end{example}

	\begin{remark}\label{rem: metric on couplings}
		As an alternative to our definition one may equivalently define first a metric on optimal couplings and then optimize over all such couplings.  More precisely, let $\Pi_{opt}(\nu) =\cup_{\mu \in \mathcal{P}(X)}\Pi_{opt}(\nu,\mu)$ be the subset of optimal couplings in $\mathcal{P}(X \times X)$ whose first marginal is $\nu$.   It is then straightforward to show that
		\begin{equation}\label{eqn: metric on couplings}
			\tilde W_\nu(\pi_0,\pi_1)=\sqrt{\inf_{\gamma \in \mathcal P(X \times X \times X)|\gamma_{yx_i}=\pi_i, i=0,1}\int_{X\times X\times X}|x_0-x_1|^2d\gamma(y,x_0,x_1)}
		\end{equation}

		is a metric on $\Pi_{opt}(\nu)$, which we will refer to as the \emph{$\nu$-based Wasserstein metric on couplings}.  Moreover, the $\nu$-based Wasserstein metric then has the following characterization:
		$$
		W_\nu(\mu_0,\mu_1)=\inf_{ \pi_i \in \Pi_{opt}(\nu,\mu_i), i=0,1}	\tilde W_\nu(\pi_0,\pi_1).
		$$ 
		
		In addition, we note that the linearized optimal transport framework defined in \cite{WangSlepcevBasuOzolekRohde13} is essentially metric \eqref{eqn: metric on couplings} between couplings.
	\end{remark}

% Recall that, in general, a curve $z_t$ parametrized by $t \in [0,1]$ in a metric space $(Z,d)$ is a \textit{minimizing geodesic} if  $d(z_s,z_t) = |s-t|d(z_0,z_1)$ for all $s,t \in [0,1]$.  It is well known that for any $\mu_0,\mu_1 \in \mathcal P(X)$, where $\mu_0$ is absolutely continuous with respect to Lebesgue measure, there is a unique Wasserstein geodesics joining them, also know as the displacement interpolant and given by $\mu_t = \Big((1-t)I+t\nabla \phi \Big)_\#\mu_0$, where $\nabla \phi$ is the Brenier map between $\mu_0$ and $\mu_1$.
 
 We pause now to describe the $\nu$-based Wasserstein metric for  several simple examples of choices for $\nu$.
\begin{example}\label{ex: ac reference measure}
	We recall the main example used to introduce generalized geodesics in Section 9.2 of \cite{AmbrosioGradientFlows2008}.  Let $\nu$ be absolutely continuous with respect to $m$-dimensional Lebesgue measure on $X$.  Then by Brenier's theorem \cite{brenier1991polar} there exist unique optimal couplings of the form $(Id,\tilde T_i)_\# \nu$ in $\Pi_{opt}(\nu,\mu_i)$, and therefore the only measure $\gamma$ with $\gamma_{yx_i} \in \Pi_{opt}(\nu,\mu_i)$ for $i=0,1$ is $\gamma=(Id,\tilde T_0, \tilde T_1)_\# \nu$.  We then have
	$$
	W_\nu(\mu_0,\mu_1) = \int_X|\tilde T_0(y)- \tilde T_1(y)|^2d\nu(y)
	$$
	so that the metric space $(\mathcal P(X),W_\nu)$ is isometric to a subset of the Hilbert space $L^2(\nu)$.  Geodesics for this metric take the form $t \mapsto \Big(tT_1 +(1-t)T_0 \Big)_\#\nu$; these are the standard generalized geodesics found in, for example, Definition 7.31 of \cite{santambook}.
\end{example}
\begin{example}
	At the other extreme, suppose $\nu =\delta_y$ is a Dirac mass.  Then for any coupling $\pi \in \Pi(\mu_0,\mu_1)$, the measure $\gamma =\delta_y \otimes \pi$ has $\gamma_{yx_i} = \delta\otimes \mu_i \in \Pi_{opt}(\nu,\mu_i)$.  Since
	$$
	\int_{X \times X \times X}|x_0-x_1|^2d\gamma(y,x_0,x_1) = \int_{X \times X}|x_0-x_1|^2d\pi(x_0,x_1) 
	$$
	we have
	$$
	W^2_\nu(\mu_0,\mu_1) = \inf_{\pi \in \Pi(\mu_0,\mu_1)}\int_{X \times X}|x_0-x_1|^2d\pi(x_0,x_1) 
	$$
	which is exactly the standard quadratic Wasserstein metric.
\end{example}
In this paper, we will be especially interested in the cases in between these extremes, when $\nu$ is singular with respect to Lebesgue measure but not a Dirac mass.  One of the simplest such cases is the following example.
\begin{example}
	Suppose that $\nu$ concentrates on a line segment and is absolutely continuous with respect to one dimensional Hausdorff measure.  It then turns out that the $\nu$-based Wasserstein metric coincides with the layerwise-Wasserstein metric introduced in \cite{KimPassSchneider19}.  The proof of this fact is slightly more involved than the previous two examples, and is included as a separate proposition below (Proposition \ref{prop: layerwise Wasserstein} -- note that the definition of the layerwise-Wasserstein metric is recalled in equation \eqref{eqn: layerwsie Wasserstein def} below as well).
\end{example}
As we show below, $W_\nu$ is also related to the following multi-marginal optimal transport problem.  Fix $\epsilon >0$ and  set

\begin{equation}\label{eqn: multi-marginal problem}
\MM_\nu^\epsilon(\mu_0,\mu_1):=\inf_{\gamma \in \Pi(\nu, \mu_0,\mu_1)} \int_{X \times X \times X} [\epsilon|x_0-x_1|^2 +|x_0-y|^2+|x_1-y|^2]d\gamma(y,x_0,x_1).
\end{equation}
The following result establishes two different characterizations of the $\nu$-based Wasserstein metric.
\begin{theorem}\label{thm: charaterization of metric}
The following holds
\begin{enumerate}
	\item$W^2_\nu(\mu_1,\mu_0)=\inf_{\pi_i \in \Pi_{opt}(\mu_i,\nu),i=0,1}\int_XW_2^2(\mu_0^y,\mu_1^y)d\nu(y)$, where $\mu_i^y$ is the conditional probability given $y$ of the optimal coupling $\pi_i=\nu(y)\otimes\mu_i^y(x) \in \Pi_{opt}(\nu,\mu_i)$ between $\nu$ and $\mu_i$.
	\item %  $W_\nu(\mu_0,\mu_1) =\frac{d}{d\epsilon}MM_\nu^\epsilon(\mu_0,\mu_1)\Big|_{\epsilon=0}$. 
	Furthermore, any weak limit point $\bar \gamma$ as $\epsilon \rightarrow 0$ of minimizers $\gamma_\epsilon$ of the multi-marginal problem \eqref{eqn: multi-marginal problem} is an optimal coupling between $\mu_0$ and $\mu_1$ for the problem \eqref{eqn: metric definition} defining $W_\nu$.% [BP; I'M NOT SURE IF $MM_{\nu}$ IS REALLY DIFFERENTIABLE, BUT IF IT IS, I THINK THIS SHOULD BE ITS DERIVATIVE.]
%	\Todo{Can we show Gamma convergence?}
\end{enumerate}
\end{theorem}
\begin{proof}
	The first part is almost immediate: for fixed $\pi_i \in \Pi_{opt}(\mu_i,\nu)$, if $\gamma \in \mathcal P(X^3)$  %is to satisfy $\gamma_{iy} =  \pi_i$, and we
	is disintegrated with respect to $\nu$, $\gamma(y,x_0,x_1) =\gamma^y(x_0,x_1)\otimes \nu(y)$, then $\gamma_{x_iy} =  \pi_i$ is equivalent to the conditional probability $\gamma^y$ on $X \times X$ having the conditional probabilities $\mu_i^y$ of the $\pi_i(x_i,y) =\mu_i^y(x_i)\otimes \nu(y)$ as marginals for almost every fixed $y$. We can therefore rewrite the integral in \eqref{eqn: metric definition} as $\int_Y\int_{X\times X}|x_0-x_1|^2d\gamma^y(x_0,x_1)d\nu(y)$.   Lemma 12.4.7 in \cite{AmbrosioGradientFlows2008} then implies that we can choose a Borel family $\gamma^y \in \mathcal{P}(X \times X)$, where for $\nu$ almost every $y$, $\gamma^y$ is an optimal coupling between the $\mu_i^y$.  The choice $\gamma(y,x_0,x_1) =\gamma^y(x_0,x_1)\otimes \nu(y)$ then minimizes the integral in \eqref{eqn: metric definition}, which yields the desired formula.

	%Wasserstein distance is then the most efficient way to do this.\marginpar{BP: more here?}

	Turning to the second point, let $\gamma$ be any competitor in the definition of $W_\nu$; that is, assume $\gamma_{yx_i} \in \Pi_{opt}(\nu,\mu_i)$ for $i=0,1$. %letting $\pi_{yx_0}$ and $\pi_{y_x0}$ be optimal couplings between $\nu$ and $\mu_0$ and $\mu_1$, respectively, the gluing lemma implies the existence of $\gamma \in P(X \times X \times X)$ such that $\gamma_{yx_i} =\pi_{yx_i}$ for $i=0,1$.  
	Since this clearly implies $\gamma \in \Pi(\nu,\mu_0,\mu_1)$, optimality of $\gamma_\epsilon$ in the multi-marginal problem yields
\begin{equation}
\label{eqn: optimal gamma epsilon}
\begin{split}
	&\int_{X \times X \times X} [\epsilon|x_0-x_1|^2 +|x_0-y|^2+|x_1-y|^2]d\gamma_\epsilon(y,x_0,x_1) \\
	&\leq \int_{X \times X \times X} [\epsilon|x_0-x_1|^2 +|x_0-y|^2+|x_1-y|^2]d\gamma(y,x_0,x_1).
\end{split}
\end{equation}

Taking the limit as $\epsilon \rightarrow 0$ gives
	$$
		\int_{X \times X \times X} [ |x_0-y|^2+|x_1-y|^2]d\bar \gamma(y,x_0,x_1) \leq \int_{X \times X \times X} [ |x_0-y|^2+|x_1-y|^2]d\gamma(y,x_0,x_1),
	$$
	or
\begin{eqnarray*}
	\int_{X \times X }  |x_0-y|^2d\bar\gamma_{yx_0}(y,x_0) +\int_{X \times X } |x_1-y|^2d\bar \gamma_{yx_1}(y,x_1)\\
	\leq\int_{X \times X }  |x_0-y|^2]d\gamma_{yx_0}(y,x_0) +\int_{X \times X } |x_1-y|^2d \gamma_{yx_1}(y,x_1)
\end{eqnarray*}
		which immediately implies the optimality of the two-fold marginals of $\bar\gamma$ in \eqref{eqn: optimal matching}, $\bar \gamma_{yx_i} \in\Pi_{opt}(\nu,\mu_i)$ for $=0,1$.
	
	%  Now let $ \gamma$ be any other  of \eqref{eqn: first level limit}.  By optimality of $\gamma_\epsilon$ we have

	Furthermore, the optimality of the two-fold marginals of $\gamma$, $ \gamma_{yx_i} \in\Pi_{opt}(\nu,\mu_i)$ means that $\int_{X \times X\times X } |x_i-y|^2d\gamma(y,x_0,x_1) \leq \int_{X \times X\times X } |x_i-y|^2d\gamma_\epsilon(y,x_0,x_1)$; combined with  \eqref{eqn: optimal gamma epsilon}, this implies that we must have, for all $\epsilon$,
	$$
		\int_{X \times X \times X} |x_0-x_1|^2 d\gamma_\epsilon(y,x_0,x_1) \leq \int_{X \times X \times X} |x_0-x_1|^2 d\gamma(y,x_0,x_1).
	$$
	Passing to the limit gives
		$$
	\int_{X \times X \times X} |x_0-x_1|^2 d\bar \gamma(y,x_0,x_1) \leq \int_{X \times X \times X} |x_0-x_1|^2 d\gamma(y,x_0,x_1).
	$$

	Since this holds for every  $\gamma$ with $ \gamma_{yx_i} \in\Pi_{opt}(\nu,\mu_i)$, it implies the desired result.

\end{proof}

	Adapting the techniques in the proof of the second part of the preceding theorem, we can also establish the following $\Gamma$-convergence result relating $MM_\nu^\epsilon$ and $W_\nu$.
\begin{proposition}
	The functional on $\mathcal{P}(X) \times \mathcal{P}(X)$ defined by
	$$
(\mu_0,\mu_1)\mapsto	F^{\epsilon}_\nu(\mu_0,\mu_1):= \frac{1}{\epsilon}\MM_{\nu}^{\epsilon}(\mu_0,\mu_1) - \frac{1}{\epsilon}W_2^2(\nu,\mu_0) -\frac{1}{\epsilon}W_2^2(\nu,\mu_0)
	$$
	$\Gamma$-converges to $(\mu_0,\mu_1)\mapsto W_\nu^2(\mu_0,\mu_1)$ as $\epsilon \rightarrow 0$ with respect to the product on $\mathcal{P}(X) \times \mathcal{P}(X)$ of the weak topology on $\mathcal{P}(X)$ with itself
\end{proposition}

\begin{proof}
	Fix $\mu_0$ and $\mu_1$ in $\mathcal{P}(X)$; to establish $\Gamma$-convergence, we must:
	\begin{enumerate}
		\item Show that whenever $\mu_0^\epsilon,\mu_1^\epsilon$ converge weakly to $\mu_0,\mu_1$ as $\epsilon \rightarrow 0$ we have 
		$$
		W_\nu^2(\mu_0,\mu_1) \leq \liminf_{\epsilon \rightarrow 0}F^{\epsilon}_\nu(\mu_0^\epsilon,\mu_1^\epsilon)  
		$$
		\item Show that there exists  $\mu_0^\epsilon,\mu_1^\epsilon$ converging weakly to $\mu_0,\mu_1$ as $\epsilon \rightarrow 0$ such that
		$$
		W_\nu^2(\mu_0,\mu_1) \geq \limsup_{\epsilon \rightarrow 0}F^{\epsilon}_\nu(\mu_0^\epsilon,\mu_1^\epsilon)  
		$$
		
		\end{enumerate}
	
	We tackle $1$ first.  Let $\mu_0^\epsilon,\mu_1^\epsilon$ converge weakly to $\mu_0,\mu_1$.  Note that for all $\epsilon$, letting $\gamma_\epsilon$ be optimal in \eqref{eqn: multi-marginal problem} with marginals $\mu_0^\epsilon$ and $\mu_1^\epsilon$, we have $\int_{X\times X\times X}|x_i-y|^2d\gamma_\epsilon(x_0,x_1,y) \geq W_2^2(\nu,\mu_i^\epsilon)$, and so
\begin{equation}\label{eqn: lower bound for epsilon}
		F^{\epsilon}_\nu(\mu_0^\epsilon,\mu_1^\epsilon) \geq \int_{X \times X \times X} |x_0-x_1|^2d\gamma_\epsilon(x_0,x_1,y).
\end{equation}
Now, let $\gamma_0$ be a weak$^*$ limit point of the $\gamma_\epsilon$.  By Theorem 5.20 in \cite{Villani-OptimalTransport-09}\footnote{Although Theorem 5.20 in \cite{Villani-OptimalTransport-09} is stated only for two marginal problems, the same result may be proven for multi-marginal problems in exactly the same way, using the equivalence between $c$-cyclical monotonicity and optimality established in \cite{Griessler18}.}, $\gamma_0$ solves \eqref{eqn: multi-marginal problem} with $\epsilon =0$; that is, multi-marginal optimal transport with cost $c_0(x_0,x_1,y) = |x_0-y|^2 +|x_1-y|^2$ and marginals $\mu_0,\mu_1$ and $\nu$.  This means that $\gamma_{x_iy}$ must be optimal for optimal transport with the quadratic cost between $\mu_i$ and $\nu$, for $i=0,1$, respectively.  Therefore, $\gamma_0$ is a competitor in the problem \eqref{eqn: metric definition} used to define $W_\nu$ and so
$$
\int_{X \times X \times X} |x_0-x_1|^2d\gamma_0(x_0,x_1,y) \geq W^2_\nu(\mu_0,\mu_1).
$$
Passing to the limit in \eqref{eqn: lower bound for epsilon}, we get
$$
		 \liminf_{\epsilon \rightarrow 0}F^{\epsilon}_\nu(\mu_0^\epsilon,\mu_1^\epsilon) \geq \int_{X \times X \times X} |x_0-x_1|^2d\gamma_0(x_0,x_1,y) \geq  W^2_\nu(\mu_0,\mu_1)
$$
along each subsequence for which $\gamma_\epsilon$ converges.  This establishes property 1.
	%as above, the righthand side converges to $W_\nu(\mu_0,\mu_1)$
Property 2 is simpler; we may take $\mu_i^\epsilon =\mu_i$ for $i=0,1$.	 To see this, note that choosing $\gamma$ in \eqref{eqn: multi-marginal problem} with optimal two-fold marginals, $\gamma_{x_iy} \in \Pi_{opt}(\mu_i,\nu)$ so that $\int_{X \times X \times X}|x_i-y|^2d\gamma(x_0,x_1,y) =W_2^2(\mu_0,\mu_1)$, causes the last  terms in $F^{\epsilon}_\nu$ to cancel, we get
	$$
	F^{\epsilon}_\nu(\mu_0,\mu_1) \leq \int_{X \times X \times X}|x_0-x_1|^2d\gamma(x_0,x_1,y).%\leq W^2_\nu(\mu_0,\mu_1)
	$$
	As this holds for each $\gamma$ with $\gamma_{x_iy} \in \Pi_{opt}(\mu_i,\nu)$, we can minimize over all such $\gamma$ to obtain
	
	$$
F^{\epsilon}_\nu(\mu_0,\mu_1)  \leq W^2_\nu(\mu_0,\mu_1).
$$
	
	Therefore,
	$$
	\limsup_{\epsilon \rightarrow 0}	F^{\epsilon}_\nu(\mu_0,\mu_1) \leq W^2_\nu(\mu_0,\mu_1)
	$$
	as desired.
\end{proof}
 
\begin{remark}\label{rem: lower dimensional convention}
In this paper, we will pay special attention to the case where $\nu$ is concentrated on a lower dimensional submanifold, parametrized by $f: Y \rightarrow X$, where $Y\subset \mathbb{R}^n$, with $n < m$, and $f$ is a smooth injection.  In this case, by a slight abuse of notation, we will often consider $\nu$ to be a probability measure on $Y$, and the quadratic optimal transport problem in the definition \eqref{eqn: optimal matching} is equivalent to optimal transport between $X$ and $Y$ with cost function $c(x,y) =|x-f(y)|^2$, or, equivalently, $c(x,y) =-x\cdot f(y)$.
\end{remark}
%\subsection{Examples}
%We pause now to describe the metric $W_\nu$ for several simple examples of choices for $\nu$. 

We next consider the case when $\nu$ concentrates on a  line; we show below that in this case, $W_\nu$ corresponds to the layerwise-Wasserstein metric from \cite{KimPassSchneider19}, whose definition is recalled below.

\begin{definition}
The layerwise-Wasserstein metric between measures $\mu_0, \mu_1 \in \mathcal P(X)$ is given by
\begin{equation}\label{eqn: layerwsie Wasserstein def}
d^2_{LW}(\mu_0,\mu_1) :=W_2^2(\mu_0^V,\mu_1^V)+ \int_0^1W_2^2(\tilde \mu_0^l,\tilde \mu_1^l)dl
\end{equation}
where the $\mu_i^V = \Big(x =(x^1,x^2,...,x^m) \mapsto x^1\Big)_\#\mu_i$  are the vertical marginals  of the $\mu_i$,  $\tilde \mu_i$ are rescaled versions of the $\mu_i$, defined by $\tilde \mu_i =(F_{\mu_i}, Id)_\#\mu_i$ where $F_{\mu_i}$ is the cumulative distribution function of $\mu_i^V$,  with the identity mapping $Id$  applied on the last $m-1$ coordinates (so that $(F_{\mu_i}, Id)(x^1,x^2,...,x^m) =(F_{\mu_i}(x^1),x^2,...,x^m)$) and $\tilde \mu_i=\mu_i^l\otimes dl$ is disintegrated with respect to its uniform vertical marginal (see \cite{KimPassSchneider19} for a more detailed description).
\end{definition}

\begin{proposition}\label{prop: layerwise Wasserstein}
Suppose $\nu$ is concentrated on the line segment $\{(t,0,0,...0): t \in \mathbb{R}\}$ and is absolutely continuous with respect to $1$-dimensional Hausdorff measure.  Then $W_\nu$ is equal to the layerwise Wasserstein metric.
\end{proposition}
\begin{proof}
Using the framework described in Remark \ref{rem: lower dimensional convention} with $n=1$ and $f(y) =(y,0,0,...,0)$, we consider optimal transport between $\mu_i$ on $X$ and $\nu$ on $Y\subseteq \mathbb{R}$ with cost function $|x_i-f(y)|^2=(x^1_i-y)^2 +(x^2_i)^2+...(x^m_i)^2$, or, equivalently, the cost function $c(x_i,y)=(x^1_i-y)^2$.  This is an \emph{index cost} in the terminology of \cite{McCannPass20}; that is, $c(x_i,y) =\tilde b(I(x_i),y)$ depends on $x_i \in \mathbb{R}^m$ only through the lower dimensional variable $I(x_i) =x_i^1$. For such costs, the optimal transport problem can be solved  semi-explicitly, essentially by solving the optimal transport problem between the one dimensional variables $I(x_i) =x_i^1$ and $y$.  We described the solution below.

% Consider optimal transport between $\mu_i$ and $\nu$, with $y=(t,0,...0)\in l$ for the quadratic cost $|x_0-y|^2 =(x^1_0-t)^2 +(x^2_0)^2+...(x^n_0)^2$.  This is an \emph{index cost} is the terminology of \cite{mccann2018optimal}, and the optimal transport problems between $\mu_i$ and $\nu$ can be solved in semi-closed form: %amounts to optimal transport between $\nu$ and the first one dimensional marginal of $\mu_0$; that is, $(P_1)_\#
 
  The level set $(T_i)^{-1}(y)$ of the optimal transport map $T_i:X \rightarrow Y$ consists of the hyperplane $\{(z_i(y),x_i^2,...,x_i^m)\}$ where the fixed $z_i(y)$ is chosen so that 
  \begin{equation}\label{eqn: monotone matching }
  \mu_i(\{(x_i^1,x_i^2,...x_i^m):x_i^1 \leq z_i(y)\}) = \nu(-\infty,y).
  \end{equation}

  	By the first part of Theorem \ref{thm: charaterization of metric}, the optimal arrangement $\gamma$ in \eqref{eqn: metric definition} then pairs the conditional probability $\mu_0^y$  on $\{z_0(y)\}\times \mathbb{R}^{m-1}$ 
  	with the corresponding conditional probability $\mu_1^y$ on $\{z_1(y)\}\times \mathbb{R}^{m-1}$. %where $z_0(y)$ and $z_1$ are chosen so that $\mu_0(\{(z_0,x_0^2,...x_0^n):x_0^1 \leq z_0\}) =\mu_1(\{(z_1,x_1^2,...x_1^n):x_1^1 \leq z_1\}) $ (meaning they are matched to the same $y$ by \eqref{eqn: monotone matching }).    
  	
  	By Theorem \ref{thm: charaterization of metric}, we then have
  	$$
  	W^2_\nu(\mu_0,\mu_1)=\int_YW_2^2(\mu_0^y, \mu_1^y)d\nu(y)
  	$$

  	Note that $|(z_0(y),x_0^2,...,x_0^m)-(z_1(y),x_1^2,...,x_1^m)|^2 = |(x_0^2,...,x_0^m)-(x_1^2,...,x_1^m)|^2 +(z_0(y)-z_1(y))^2$, so that $W_2^2(\mu_0^y, \mu_1^y) = W_2^2( \hat \mu_0^y,  \hat \mu_1^y) +(z_0(y)-z_1(y))^2$, where the measures $\hat \mu_i^y := \Big((z_i(y),x_i^2,...,x_i^m) \mapsto (x_i^2,...,x_i^m)\Big)_\#\mu_i^y$ are measures on $\mathbb{R}^{m-1}$. 	Letting $F$ be the quantile function of  $\nu$, and changing variables via $l=F(y)$, we have
  	\begin{eqnarray*}
  		W^2_\nu(\mu_0,\mu_1)&=&\int_Y [W_2^2(\hat \mu_0^y, \hat \mu_1^y) +(z_0(y)-z_1(y))^2]d\nu(y)\\
  		&=&\int_0^1 [W_2^2(\hat \mu_0^{F^{-1}(l)}, \hat \mu_1^{F^{-1}(l)}) +(z_0(F^{-1}(l))-z_1(F^{-1}(l))^2]dl
  	\end{eqnarray*}
  	Note that the second term is exactly the Wasserstein metric between the first marginals of $\mu_0$ and $\mu_1$.  Since the conditional probabilities $\tilde \mu^l_i$ are both supported on $\{l\} \times \mathbb{R}^{m-1}$, and $\Big((l,x_i^2,...,x_i^m) \mapsto (x_i^2,...,x_i^m)\Big)_\#\tilde \mu_i^l =\hat \mu_i^{F^{-1}(l)}$, we have $W_2^2(\tilde \mu^l_i, \tilde \mu^l_i) = W_2^2(\hat \mu_0^{F^{-1}(l)}, \hat \mu_1^{F^{-1}(l)}) +(l-l)^2 =W_2^2(\hat \mu_0^{F^{-1}(l)}, \hat \mu_1^{F^{-1}(l)})$.  The last line above is then exactly the definition of the layerwise-Wasserstein metric.
\end{proof}
%\begin{proposition}
%	Let the dimension $n=2$ and assume $\nu$ is concentrated on a line segment $l$ and is absolutely continuous with respect to one dimensional Hausdorff measure.  Then the projection $\gamma_{01}:=((x_0,x_1,y)\mapsto (x_0,x_1))_\#\gamma$ of the optimal $\gamma$ is equal to $(Id,T)_\#\mu_0$, where $\mu_0$ is the Knothe-Rosenblatt transport from $\mu_0$ to $\mu_1$, with respect to the orthonormal basis $\{e_l, e_l^\perp\}$, where $e_l$ is the unit vector in the direction of $l$.
%\end{proposition}
%\begin{proof}
%Without loss of generality, $e_l =(1,0)$ and $l$ passes through the origin.  Optimal transport between $\mu_i$ and $\nu$, with $y=(t,0)\in l$ for the quadratic cost $|x_0-y|^2 =(x^1_0-t)^2 +(x^2_0)^2$ amounts to optimal transport between $\nu$ and the horizontal marginal of $\mu_0$; the level set $(T_0)^{-1}(y)$ of the optimal map consists of the line segment $\{(z_0,x_0^2)\}$ where the fixed $z_0$ is chosen to that $\mu_0(\{(x_0^1,x_0^2):x_0^1 \leq z_0\}) = \nu(-\infty,y)$.  

%By the first part of Theorem \ref{thm: charaterization of metric}, the optimal arrangement $G$ then pairs the conditional probability  on $\{(z_0,x_0^2)\}$ with the corresponding conditional probability on $\{(z_1,x_1^2)\}$, where $z_0$ and $z_1$ are chosen so that $\mu_0(\{(x_0^1,x_0^2):x_0^1 \leq z_0\}) =\mu_1(\{(x_1^1,x_1^2):x_1^1 \leq z_1\}) $ (meaning they belong they are matched to the same $y$).  The matching between these conditional distributions is done in an optimal (and therefore monotone) way; this is exactly the Knothe-Rosenblatt rearrangement.
%\end{proof}
As was noted in \cite{KimPassSchneider19}, in two dimensions the layerwise-Wasserstein metric corresponds to the Knothe-Rosenblatt rearrangement \cite{Knothe57,Rosenblatt52} and so we immediately obtain the following.
\begin{corollary}
	Let $m=2$.  Then, under the assumptions in the preceding Proposition, the optimal rearrangement $\gamma$ in \eqref{eqn: metric definition} satisfies $\gamma_{x_0x_1} =(Id,G)_\#\mu_0$, where $G$ is the Knothe-Rosenblatt rearrangement.
\end{corollary}
\begin{proof}
	For each $y$, the level set of the optimal map $T^{-1}(y)$ is the line segment $\{(z_i(y),x_i^2)\}$ where \eqref{eqn: monotone matching } becomes:
	
	  \begin{equation}\label{eqn: 2d monotone matching }
		\mu_i(\{(x_i^1,x_i^2):x_i^1 \leq z_i(y)\}) = \nu(-\infty,y).
	\end{equation}
Thus, any matching $\gamma$ with $\gamma_{yx_i} \in \Pi_{opt}(\nu,\mu_i)$ must couple $y$ with $\{(z_i(y),x_i^2)\}$; it follows that $\gamma_{x_0x_1}$ must couple  $\{(z_0(y),x_0^2)\}$ to $\{(z_1(y),x_1^2)\}$.  Since \eqref{eqn: 2d monotone matching } implies $\mu_0(\{(x_0^1,x_0^2):x_0^1 \leq z_0(y)\}) =\mu_1(\{(x_1^1,x_1^2):x_1^1 \leq z_1(y)\})$, this means that $\gamma_{x_0x_1}$ is montone in the first coordinate.  The optimal way to couple the conditional probabilites, concentrated on $\{(z_0(y),x_0^2)\}$ and $\{(z_1(y),x_1^2)\}$ is then monotonic; this exactly characterizes the Knothe-Rosenblatt rearrangement.
\end{proof}
\begin{remark}
%More generally, when $n\geq 3$, $\nu$ is concentrated on a line segment and is absolutely continuous with respect to $1$-dimensional Hausdorff measure, the metric $W^2_\nu$ is (up to a rotation) the layerwise-Wasserstein metric  from \cite{KimPassSchneider20}.  

One can recover a similar characterization of the Knothe-Rosenblatt rearrangement in higher dimensions by looking at iterated optimization problems along an orthogonal basis.  This corresponds to the limit of a multi-marginal problem where the interactions are weighted iteratively, as we show in Appendix \ref{sect: multiple reference measures}.  This is closely related to the main result in \cite{CarlierGalichonSantambrogio10}, where the Knothe-Rosenblatt rearrangement is characterized as a limit of optimal transport maps with anisotropic costs.
	\end{remark}
\subsection{Dual metrics}\label{sect:dual metrics}

We now define a class of  metrics on a lower dimensional space which are \textit{dual} to $W_\nu$ in a certain sense; more precisely, they use Kantorovich potentials $v(y)$ on the variable $y$ arising from optimal transport to a fixed reference measure $\mu(x)$ to compare measures on the $y$ variable, whereas $W_\nu$ uses in part the Kantorovich potential $v(y)$ for optimal transport from a reference measure $\nu(y)$ on $y$ to measures on $x$ in order to compare those free measures.
	% meaning they mainly compare two given measures through the Kantorovich potentials, solution the dual optimal transport problem between each of them and a reference measure.}

Fix reference measures $\mu \in \mathcal{P}(X)$ and  $\sigma \in \mathcal{P}(Y)$ on bounded domains $X \subseteq \mathbb{R}^m$ and $Y\subseteq \mathbb{R}^n$, with $n \leq m$, absolutely continuous with respect to $m$ and $n$ dimensional Lebesgue measure, respectively.  We let $c:\bar X\times \bar Y\rightarrow \mathbb{R}$ be a cost function which is $C^1$ up to the boundary and satisfies the \emph{twist condition}; that is, for each fixed $x \in X$, the mapping from $\bar Y$ to $\mathbb{R}^m$ defined by 
	$$
 	y \mapsto D_xc(x,y)
	$$
	is injective.  Note that $c \in C^1 (\bar X \times  \bar Y)$ implies that $c$ is Lipschitz, since $\bar X$ and $\bar Y$ are compact.

We let $\mathcal P_{ac,\sigma}(Y)$ be the set of probability measures on $Y$ which are absolutely continuous with respect to $\sigma$. For measures $\nu_0, \nu_1 \in \mathcal P_{ac,\sigma}(Y)$, we define
\begin{equation}\label{eqn: dual metric}
W^*_{\mu, \sigma,c, p}(\nu_0,\nu_1):=|| D v_0- D v_1 ||_{L^p(\sigma)}
\end{equation}%\marginpar{BP: Should we introduce this for general costs?  I think with the twist condition, $v$ uniquely identifies $\nu$, so we still get a metric.\textcolor{red}{LN: Yes, you are right...probably we can add a c to the indices and putting the twist conditions as hypothesis}}
for some $p \in [1, \infty]$ where $v_i$ is the $c$-concave Kantorovich potential (that is, $c$-conjugate solutions to \eqref{eqn: ot dual} corresponding to the optimal transport problem between $\mu$ and $\nu_i$ with cost function $c$).

\begin{proposition}
$W^*_{\mu, \sigma, c, p}$ is a metric on  $\mathcal P_{ac,\sigma}(Y)$.
\end{proposition}
\begin{proof}
	Since the $L^p$ norm clearly induces a metric, all that needs to be proven is that for each measure $ \nu \in \mathcal P_{ac,\sigma}(Y)$, the gradient $Dv$ of the  Kantorovich potential is in $L^p(\sigma)$, and that the mapping $\nu \mapsto Dv$ is a bijection.
	Since the cost function is Lipschitz, a now standard argument originating in \cite{McCann2001}, implies that the potential  $v$ is Lipschitz; absolute continuity of $\nu$ with respect to $\sigma$ (and therefore Lebesgue measure) ensures that it  exists %and is uniquely determined 
	$\sigma$ almost everywhere, by Rademacher's theorem. Furthermore, since $v$ is semi-convex and Lipschitz, the gradient $Dv$ is also bounded, and so $Dv \in L^p(\sigma)$ for all $ p\geq 1$.
	
	 %Absolute continuity of $\mu$
	 %means that the potential $u$ on the $x$ side is uniquely determined by $\nu$;
	 Absolute continuity of $\nu$ with respect to $\sigma$, and therefore Lebesgue measure, ensures that $Dv$ is uniquely determined by $\nu$ \cite[Proposition 1.15]{santambook}.
	 
	 On the other hand, it is well known that the twist condition ensures that the unique optimizer to the optimal transport problem between $\mu$ and $\nu$ concentrates on a graph, $y=T(x)$, where $T$ is uniquely determined by the gradient of the potential $u(x)$ via the equation $Du(x) =D_xc(x,T(x))$.  Since $\nu= T_\#\mu$, we have that $\nu$ is uniquely determined by $T$ and therefore $u$.
	 %the pushforward equation $\nu =D u_\#\mu$ clearly implies that $\nu$ is determined by $u$. 
	 The $c$-concave Kantorovich potential  $v(y)=\inf_{x \in X}[c(x,y)-u(x)]$ is then clearly in one to one correspondence with $u(x)=\inf_{y \in Y}[c(x,y)-v(y)]$ and, therefore, it uniquely determines $\nu$.  Thus $\nu$ and $Dv$ are in one-to-one correspondence, as desired.
	  % $\nu_i =D (v_i)^*_\#\mu$ is uniquely determined by $v_i$, 
	
\end{proof}

%the Monge condition ensures that the Kantorovich potential uniquely determines the $\nu_i$,
%so this is a well defined metric on the set of probability measures which are absolutely continuous to $\sigma$.  
To help explain the motivation behind this metric, it is useful to first consider the $m=n$ case with $c(x,y) =|x-y|^2$.  The metric $W^*_{\mu, \sigma,c, 2}$ is then similar to the $\nu$-based Wasserstein metric.  Recall from Example \ref{ex: ac reference measure} that in this case $W_\nu(\mu_0,\mu_1)$ between two measures in $\PP(X)$ coincides with the $L^2$ metric on the optimal maps from $\nu$ to the $\mu_i$; on the other hand, $W^*_{\mu, \sigma,c, 2}(\nu_0,\nu_1)$ for two measures in $\PP(Y)$ coincides with the $L^2$ metric on the optimal maps from the $\nu_i$ to $\mu$.  That is, $W^*_{\mu, \sigma,c, 2}(\nu_0,\nu_1)$ compares gradients of potential on  the opposite, rather than same, side of the problem from the reference measure, and so is in some sense dual to $W_\nu$ in  Example \ref{ex: ac reference measure}.

The case where $m>n$ is an attempt to generalize this idea in the same way that the $\nu$-based Wasserstein metric generalizes the $L^2$ metric in Example \ref{ex: ac reference measure}.  We emphasize that the motivating example  is when $c=|x-f(y)|^2$ is quadratic between $x$ and $f(y)$ in a lower dimensional submanifold, using the convention in Remark \ref{rem: lower dimensional convention}.  In this case, $W^*_{\mu, \sigma, c, p}$ is dual to the $\nu$ based Wasserstein metric, in the sense that it relies on the Kantorovich potential on the target, rather than source, side.

The metric $W^*_{\mu, \sigma, c, p}$ arises naturally in certain variational problems among measures on $Y$ involving optimal transport to the fixed measure $\mu$ on $X$, which is more naturally described by the potentials $v$ on $Y$ than $u$ on $X$.  We expand on this viewpoint in section \ref{sect: fixed point} below, where we establish a fixed point characterization of solutions to a class of these problems. In the proof, we actually use a metric which corresponds to using the mass splitting $k$, introduced in \cite{McCannPass20} and recalled in equation \eqref{eqn: mass splitting} below, rather than potentials (these are the same when the problem satisfies the nestedness condition introduced in \cite{McCannPass20}, as is the case for the optimal $\nu$ under appropriate conditions identified in \cite{NennaPass1} and recalled in \eqref{eqn: suff condition for nestedness} below, but not in general).
%\marginpar{BP: It might make sense to move this to an appendix, together with equivalence to a corresponding metric (with several, rather than 2, iterated optimizations).}
%\begin{proposition}
%	The unique $W_{\mu,\sigma,2}^*$ geodesic between $\nu_0$ and $\nu_1$ is given by $\Big((1-t)v_0+tv_1\Big)_\# \mu$, where $v_i$ is the Kantorovich potential for  
%\end{proposition}
\section{Unequal dimensional optimal transport}
 
When the base measure $\nu$ concentrates on a lower dimensional submanifold, the definition \eqref{eqn: metric definition} of the $\nu$-Wasserstein metric relies on unequal dimensional optimal transport.  We briefly review the known theory on this topic here, and establish a lemma which we will need later on.

For $c(x,y) =-x\cdot f(y)$ as in Remark \ref{rem: lower dimensional convention}, $y \in Y$ and $p \in-P_y$, where $-P_y:=\{-x\cdot Df(y) \in X:x \in X \}$, we define, as in \cite{McCannPass20}, noting that $D_yc(x,y) =-x\cdot Df(y)$ 
\begin{equation}\label{eqn: def of X1}
X_1(y,p):=\{x\in X: -x\cdot Df(y) =p\}.
\end{equation}
Note that, since our cost is linear in $X$,  each $X_1(y,p)$ is an affine $(m-n)$-dimensional submanifold of $X$, provided that $Df(y)$ has full rank.  For a fixed $y$, each $X_1(y,p)$ is parallel.  %We also recall that, for optimal transport between $\mu$ on $X$ and $\nu$ on $Y$, there is a set $X(u) =X(v^c)$ of full measure depending on the Kantorovich potential $u(x)$ on which the change of variables 

The following result shows that when we disintegrate the optimizer with respect to $\nu$, the conditional probabilities are absolutely continuous with respect to $(m-n)$-dimensional Hausdorff measure, and provides a formula for their densities.
\begin{theorem}\label{thm: structure of multi-to one-dimensional transport}
Consider optimal transport between probability measures $\mu$ on $X \subseteq \mathbb{R}^m$ and $\nu$ on $Y \subseteq\mathbb{R}^n$, where $m>n$, both absolutely continuous with respect to Lebesgue measure,  with densities $\bar \mu(x)$ and $\bar \nu(y)$, respectively,  and cost function $c(x,y)=-x\cdot f(y)$ where $f \in C^2(Y)$ is injective and non-degenerate, that is $Df(y)$ has full rank everywhere.

Then there is a unique optimal measure $\pi=(Id,T)_\#\mu$ in \eqref{eqn: ot with general cost}, which
 concentrates on the graph of a function $T:X \rightarrow Y$.  The Kantorovich potential $v(y)$ is  differentiable for almost every $y$, and at each such point $T^{-1}(y)$ is concentrated on the affine manifold $X_1(y,Dv(y))$. The conditional probability $\mu^y(x) $ %(which is the marginal of $\pi^y(x) $ from $\pi(x,y) = \pi^y(x) \otimes \nu(y)$) 
corresponding to the disintegration $\mu=\mu^y\otimes \nu(y)$ of $\mu$ with respect to $T$ is absolutely continuous with respect to $(m-n)$ - dimensional Hausdorff measure on $X_1(y,Dv(y))$ for almost every $y$.  Furthermore, there is a subset $X(v^c) \subset X$ of full volume such that this conditional probability is concentrated on $T^{-1}(y) \cap X(v^c)$ with density $\bar \mu^y(x)$  with respect to $(m-n)$ -- dimensional Hausdorff measure given by 
\begin{equation}
\bar \mu^y(x) = \bar \mu(x)/[\bar \nu(y)JT^y(x)]
\end{equation}
where $JT^y(x) = \frac{\sqrt{\det [Df(y)Df^T(y)]}}{\det [-x\cdot D^2f(y)-D^2v(y)]}$ is the $n$-dimensional Jacobian of the optimal transport map defined by $T(x) =(f)^{-1}(-Dv^c(x))$, evaluated at any point $x \in T^{-1}(y)$.
\end{theorem}
\begin{proof}The existence and uniqueness of an optimal transport map $T$ is well known.  
The remainder of the proof is based on the  proof of Theorem 2 in \cite{McCannPass20}.  We repeat the preliminary setup of that argument for the reader's convenience.
 Letting $u(x)=v^c(x)$ and $v(y)=u^c(y)$ be the Kantorovich potentials, it is well known that $u$ is semi-concave and therefore its differential $Du(x)$ exists almost everywhere and is locally of bounded variation (see Theorem 6.8 in \cite{EvansGariepy15}).  Theorem 6.13 in \cite{EvansGariepy15} then implies that for each integer $i>0$ there exists a set $Z_i$ of Lebesgue measure less than $1/i$, and a continuously differentiable map $H_i$ with $DH_i =Du$ and $D^2H_i=D^2u$ outside of $Z_i$.  The $Z_i$ may be taken to be nested, that is, $Z_{i+1} \subset Z_i$, and we set $Z_\infty = \cap_{i=1}^{\infty}Z_i$ and $X(v^c) =dom(D^2u)\setminus Z_\infty$.
	
Now set $\tilde Y =dom(D^2v)$ and  $\tilde X = T^{-1}(\tilde Y)$.  We note that $u(x) +v(y)-c(x,y) \leq 0$, with equality at $(x,T(x))$.   For $x \in \tilde{X}$, differentiating with respect to $y=T(x) \in \tilde Y$ yields $Dv(y) =x\cdot Df(y)$, or $Dv(T(x)) = x\cdot Df(T(x))$.  If $x \in dom(DT)$, differentiating with respect to $x$ then yields:

$$
[D^2v(T(x)) -x^tD^2f(y)]DT(x) =Df(T(x)).
$$
As the right hand side has full rank by assumption, this means that the $m\times n$ matrix $DT(x)$ must also have full rank, and so its Jacobian $JT(x)=det^{1/2}[DT^t(x)DT(x)] >0$ must be non-zero.  
We then set $\tilde X_i =\{x \in\tilde X \setminus Z_i: |x| \leq i, JT(x) >1/i \}$ and $X_i=\tilde X_i \setminus \tilde X_{i-1}$, so that $\tilde X \setminus Z_{\infty} =\cup_{i=1}^\infty X_i$ and the union is pairwise disjoint,

% We define the sets $X_i$ for each $i=1,2,...$ and $X(v^c)$ as in that paper, and take $\bar \mu_i = \bar \mu 1_{X_i}$ (note that the density in \cite{McCannPass20} was denoted by $f$); in our context, since both source and target measures are absolutely continuous, this means that restricted to each $X_i$, $T=T_i$ and $DT=DT_i$ for  $C^1$ maps with $JT(x) >1/i$ and that $T_i$ has a Lipschitz extension to all of $\mathbb{R}^m$.  The $X_i$ are pairwise disjoint and their union $\cup_{i=1}^\infty X_i = X(v^c) =T^{-1}(dom(D^2v)) \setminus Z_\infty$, where $X(v^s) = dom(D^2v^s)\setminus Z_\infty$ for a certain set $Z_\infty$ (the full details are found in \cite[Theorem 2]{McCannPass20}).

Now, equation (30) in \cite{McCannPass20} implies that for any $h \in L^\infty(\mathbb{R}^m)$ we have
\begin{eqnarray}
\int_{X_i}h(x) \bar \mu_i(x) dx&=&\int_Y	\bar \nu(y)\frac{1}{\bar \nu(y)}dy\int_{T_i^{-1}(y) \cap X_i}\frac{h(x)\bar \mu_i(x)}{JT^y_i(x)}d\mathcal H^{m-n}(x)\\
&=&\int_Y	\bar \nu(y)\frac{1}{\bar \nu(y)}dy\int_{T_i^{-1}(y) \cap X(v^s)}\frac{h(x)\bar \mu_i(x)}{JT_i^y(x)}d\mathcal H^{m-n}(x)
\end{eqnarray}
One can then sum on $i$ to obtain
\begin{equation}\label{eqn: change of variables}
\int_{X}h(x) \bar \mu(x) dx=\int_Y	\bar \nu(y)dy\int_{T^{-1}(y) \cap X(v^s)}\frac{h(x)\bar \mu(x)}{\bar \nu(y)JT^y(x)}d\mathcal H^{m-n}(x).
\end{equation}
As this holds for every $h \in L^{\infty}(\mathbb{R}^m)$, it implies the desired disintegration.
\end{proof}
\subsection{Nestedness and its high dimensional generalization}
In what follows, in Section \ref{sect: fixed point} in particular,  we will sometimes consider optimal transport between $\mu$ and $\nu$ with respect to a general cost, $c(x,y)$ on $X \times Y \subseteq \mathbb{R}^m \times \mathbb{R}^n$, with $m \geq n$.  In this case, we will denote the analogue of  $X_1(y,p)$ by 
\begin{equation}
X^c_1(y,p):=\{x\in  X: D_yc(x,y) =p\}.
\end{equation}

 Assuming that $c$ is non-degenerate, so that the $m \times n$ matrix $D^2_{xy}c(x,y)$ has full rank $n$, the implicit function theorem implies that $X^c_1(y,p)$ is an $m-n$ dimensional submanifold of $X$. 
If we denote by $v(y)$ the Kantorovich potential, we have
$$
c(x,y)-v(y) \geq  u(x)=\inf_{y' \in Y}(c(x,y') -v(y'))
$$
with equality $\gamma$ almost everywhere.  For $(x,y)$ such that equality holds, this implies that $y' \mapsto c(x,y')-v(y')$ is minimized at $y'=y$.

 At any such point where $v$ is differentiable, we therefore obtain $Dv(y)=D_yc(x,y)$, so that
$$
 \partial^cv(y) :=\{x \in X: c(x,y)-v(y) =\inf_{y' \in Y}(c(x,y') -v(y'))\}\subseteq X^c_1(y,Dv(y)).
$$

The generalized nestedness condition, introduced in \cite{McCannPass20} is that this inclusion is actually an equality.

\begin{definition}\label{def: generalized nestedness}
	We say that $(c,\mu,\nu)$ satisfies the \emph{generalized nestedness condition} if for all $y$ at which $v$ is differentiable, we have
	$$
	\partial^cv(y) =X^c_1(y,Dv(y))
	$$
	where $v$ is the $c$-conjugate Kantorovich potential for transport between $u$ and $v$.
\end{definition}
The importance of this condition,  uncovered in \cite{McCannPass20}, is as follows.  Solving the optimal transport problem \eqref{eqn: ot with general cost} is essentially equivalent to solving the dual problem \eqref{eqn: ot dual}.  In turn, it is shown in \cite{McCannPass20} that the solution $(v^c,v)$ to \eqref{eqn: ot dual} can be recovered by solving a partial differential equation for $v$, which is in general \emph{non-local} and therefore extremely challenging.  Under the generalized nestedness condition, however, it turns out that this equation reduces to a local equation, and therefore becomes much more tractable.

%If the inclusion above is an \emph{equality} for all $y \in Y$, we say that $(c,\mu,\nu)$ satisfies the \emph{generalized nestedness condition}.  
The origin of the term generalized nestedness comes from the case when $Y=(\underline y, \overline y)$ is one dimensional, in which case it can be shown that generalized nestedness is  essentially equivalent to nestedness of the corresponding superlevel sets.  More precisely, we define:
\begin{equation}
	X^c_\geq(y,p):=\{x\in X: D_yc(x,y) \geq p\}.
\end{equation} 
Assuming that $\mu$ and $\nu$ are both absolutely continuous, with densities bounded above and below, for each fixed $y$, we note that the measure $\mu(X^c_\geq(y,k))$ of the super level set 	$X^c_\geq(y,k)$ is a decreasing function of $k$.  There will then be a unique $k =k(y)$ satisfying the   population balancing condition:
\begin{equation}\label{eqn: mass splitting}
\mu(X^c_\geq(y,k)) = \nu((\underline y,y)).
\end{equation}
The definition of nestedness found in \cite{PassM2one} is then the following.
\begin{definition}\label{def: nestedness}
	Assume that $Y$ is one dimensional.  We say the model $(c,\mu,\nu)$ is \textit{nested} if $$
	X^c_{\geq}(y,k(y))\subseteq X^c_{\geq }(y',k(y'))
	$$
	whenever $y \leq y'$.
\end{definition}
The relevance of the nestedness condition is that it allows one to almost explicitly solve \eqref{eqn: ot with general cost} \cite{PassM2one}. 
One can  attempt to define a mapping $T: X \rightarrow Y$ that maps each $x \in X_1^c(y,k(y))$ to $y$.  In general, this may not result in a well defined mapping, since if  $ X_1^c(y,k(y))$ intersects $ X_1^c(y',k(y'))$ for $y \neq y'$, the procedure attempts to map each $x$ in the intersection to \emph{both} $y$ and $y'$.  It turns out that the nestedness condition ensures that this does not happen; furthermore, not only is the mapping $T$ described above well defined for nested models, but the measure $(Id,T)_\#\mu$ is the solution to the Kantorovich problem \eqref{eqn: ot with general cost} and the $k(y)$ defined above coincides with the derivative $v'(y)$ of the Kantorovich potential.
%It is proven in \cite{ChiapporiMcCannPass15p} that the solution to the problem matches each $x \in X_1^c(y,k)$ with $y$, where $k=k(y)$ chosen to satisfy the

Before closing this section, we illustrate the nestedness condition with the following simple example from \cite{PassM2one}.  Several other examples may be found in \cite{PassM2one} \cite{ChiapporiMcCannPass15p}\cite{ChiapporiMcCannPass19}.
\begin{example}	
Let $c(x,y) =-x\cdot (\cos(y),\sin(y)) =-x_1\cos(y) -x_2\sin(y)$, $\mu$ be uniform on the subset $X=\{x: 0<r <1, 0<\theta<\bar \theta  \}\subset \mathbb{R}^2$, of the unit ball in $\mathbb{R}^2$, expressed in polar coordinates, and $\nu$ be uniform on $Y=(0,\bar \theta)$.  This corresponds to to optimal transport with the quadratic cost between $X$ and the corresponding circular arc, parametrized by the angle $y$.  The solution clearly maps each $x$ radially outward to $x/|x|$, so the optimal map is $y=T(x) =\arctan(x_2/x_1)$. We will illustrate the nestedness concept by considering two cases; in the first case, the model is nested, and so the procedure described above successfully constructs the optimal mapping.  In the second, the model is not nested, and the procedure fails to produce a well defined mapping.

\textbf{Case 1:} Suppose that $\bar \theta <\pi$.  Then  $	X^c_1(y,k)=\{X \in X: x\cdot(\sin(y),-\cos(y)) =k\}$ is a line segment parallel to $(\cos(y),\sin(y))$ and  $	X^c_\geq(y,k)$ is the portion of $X$ lying below this line segment.  The balancing condition \eqref{eqn: mass splitting} is satisfied for $k=k(y) =0$, in which case  $	X^c_1(y,k)$ passes through the origin and the mass of the wedge $	X^c_\geq(y,k) =\{x: 0<r <1, 0<\theta<y  \}$ satisfies
$$
\mu(X^c_\geq(y,k)) =y/\bar \theta =  \nu((0,y)).
$$
These regions $X^c_\geq(y,k(y)) =X^c_\geq(y,0)$ are nested, 	$X^c_\geq(y,0) =\{x: 0<r <1, 0<\theta<y  \} \subset \{x: 0<r <1, 0<\theta<y'  \}= X^c_\geq(y',0) $ for $y<y'$, and indeed, mapping  each $x \in X^c_1(y,0)$ to $y$ coincides with optimal map, while setting $v'(y) =k(y)$ yields the Kantorovich potential.

\textbf{Case 2:} Suppose now that $\pi<\bar \theta <2\pi$.  Then for $y < \bar \theta -\pi$, note that the line segment  $	X^c_1(y,0)$ through the origin now intersects $X$ in the quadrant $\{x_1,x_2<0\}$ and so $	X^c_\geq(y,0) =\{x: 0<r <1, 0<\theta<y  \} \cup \{x: 0<r <1, \pi +y<\theta<\bar \theta  \}$; we therefore have $\mu(X^c_\geq(y,0)) =y/\bar \theta +\bar \theta -\pi -y >y/\bar \theta =\nu((0,y))$.  The $k=k(y)$ satisfying \eqref{eqn: mass splitting} must therefore be larger than $0$.  Assigning each $x$ in $X^c_=(y,k(y))$ to $y$ then clearly does not correspond to the optimal map, and so the result in \cite{PassM2one} implies that the model cannot be nested.  Alternatively, this can be seen directly by noting that $X^c_\geq(y,k(y))$ contains some point in the quadrant $\{x_1,x_2<0\}$ for $y < \bar \theta -\pi$, whereas for  $  \bar \theta -\pi<y'<\pi$, $k(y')=0$ as before, so $X^c_\geq(y',k(y')) =\{x: 0<r <1, 0<\theta<y  \}$ is contained in the region where $x_2>0$. 
\end{example}

\subsection{Sufficient conditions for nestedness}
In our previous work \cite{NennaPass1}, motivated by problems in economics and game theory, we considered the problem of minimizing among probability measures in $\mathcal P([\underline y, \overline y])$ a class of functionals including:
\begin{equation}\label{eqn: variational problem}
\nu \mapsto T_c(\mu,\nu) + \int \bar \nu(y) \log(\bar \nu(y))dy,
\end{equation}
for a fixed measure $\mu$ on the bounded domain $X \subseteq \mathbb{R}^m$, absolutely continuous with respect to Lebesgue measure, and a fixed domain $Y:=(\underline y, \overline y) \subset \mathbb{R}$.     Here $T_c(\mu,\nu):=\inf_{\pi \in \Pi(\mu,\nu)}\int_{X \times Y}c(x,y)d\pi(x,y)$ is the cost of optimal transport between $\mu$ and $\nu$ with respect to the cost function $c(x,y)$ and $\bar \nu$ is the density of $d\nu(y) =\bar \nu (y)dy$ with respect to Lebesgue measure (if $\nu$ is not absolutely continuous, the functional is taken to be infinite).

 In \cite{NennaPass1}, we showed that the solution to \eqref{eqn: variational problem} is nested provided that:

\begin{equation}\label{eqn: suff condition for nestedness}
\sup_{y_1 \in \Y,y_0\leq y\leq y_1} 	\frac{D^{min}_\mu(y_0,y_1,k(y_0))}{y_1-y_0}-M_c\frac{e^{-M_cy}}{e^{M_c\overline{y}}-e^{M_c\underline{y}}} -1<0
\end{equation}
for all $y_0 \in \Y$ 
Here, $k(y_0)$ is defined by \eqref{eqn: mass splitting},

$$
D^{min}_\mu(y_0,y_1,k_0) =\mu(X_\geq(y_1,k_{max}(y_0,y_1,k_0))\setminus X^c_\geq(y_0,k_0))
$$
is the minimal mass difference where $k_{max}(y_0,y_1,k_0)=\sup\{k: X_\geq(y_0,k_0) \subseteq X_\geq (y_1,k)\}$ and $M_c :=\sup_{(x,y_0,y_1) \in (\bar \X \times \bar \Y \times \bar Y)}\frac{|c(x,y_0) -c(x,y_1)|}{|(x,y_0) -(x,y_1)|}$ is the Lipschitz norm with respect to variations in $y$ (note that the condition has been reformulated slightly from \cite{NennaPass1}, to match the choice of functional in \eqref{eqn: variational problem}; in \cite{NennaPass1} more general functionals were studied).

 \section{Uniqueness of optimal couplings and structure of geodesics when $\nu$ concentrates on a lower dimensional submanifold}

We note that geodesics for the metric $W_\nu$ are generalized geodesics with base $\nu$, according to the general definition in \cite{AmbrosioGradientFlows2008}.  When $\nu$ is sufficiently regular (for instance, absolutely continuous with respect to Lebesgue measure, as in Example \ref{ex: ac reference measure}), generalized geodesics are \emph{uniquely} determined by $\mu_0$ and $\mu_1$. In this case, $\Pi_{opt}(\nu, \mu_i)$ consists of a single measure concentrated on the graph of a function $\tilde T_i:X\mapsto X$ pushing $\nu$ forward to $\mu_i$; $\gamma=(Id,\tilde T_0,\tilde T_1)_\#\nu$ is then the unique measure satisfying $\gamma_{yx_i} \in \Pi_{opt}(\nu,\mu_i)$, and therefore the unique coupling minimizing the integral in the definition of $W_{\nu}(\mu_0,\mu_1)$.  If $\nu$ is singular, this reasoning does not apply; existence of minimizers in \eqref{eqn: metric definition} follows immediately from standard continuity-compactness results, but they may be non-unique in general.  

Here we specialize to the case when $\nu$ concentrates on some lower dimensional submanifold.  Given Proposition \ref{prop: layerwise Wasserstein}, we can regard this as a sort of non-linear generalization of the layerwise-Wasserstein metric.  In this setting, if the $\mu_i$ are absolutely continuous with respect to Lebesgue, we obtain a uniqueness result as a consequence of Theorem \ref{thm: charaterization of metric} and Lemma \ref{lem: structure of geodesics}.  We also characterize geodesics with respect to our metric is this setting.
\begin{theorem}\label{thm: graphical solutions}
	Assume that both $\mu_i$ are absolutely continuous with respect to Lebesgue measure and that $\nu$ is supported on a smooth $n$-dimensional submanifold of $X$, and is absolutely continuous with respect to $n$-dimensional Hausdorff measure. Then the optimal coupling $\gamma$ in \eqref{eqn: metric definition} is unique and its two-fold marginal $\gamma_{x_0x_1}$ is induced by a mapping $G:X \rightarrow X$ pushing $\mu_0$ forward to $\mu_1$.  Letting $T_i$ denote the optimal maps from $\mu_i$ to $\nu$, the restriction of $G$ to  $T^{-1}_i(y)$ is the optimal mapping pushing $\mu_0^y$ forward to $\mu_1^y$ for a.e. $y$, where, as in  Theorem \ref{thm: charaterization of metric}, $\mu_i(x) =\mu_i^y(x) \otimes \nu(y)$.  This mapping takes the form $\nabla \phi^y :  X_1(y,Dv_0(y)) \rightarrow X_1(y,Dv_1(y)) $, where $\phi^y: X_1(y,Dv_0(y)) \rightarrow \mathbb{R}$ is a convex function defined on the $(m-n)$-- dimensional affine subset $X_1(y,Dv_0(y))$.
\end{theorem}
\begin{proof}
	Brenier's theorem  implies that the optimal $\pi_i =(Id,T_i)_\sharp\mu_i$ between $\mu_i$ and $\nu$
		is unique and concentrated on the graph of a function $T_i$ for $i=0,1$ \cite{brenier1991polar}; the conditional probabilities $\mu_i^y$ are concentrated on $T_i^{-1}(y)$ which are clearly pairwise disjoint.  Using the first characterization in Theorem \ref{thm: charaterization of metric}, it is enough to show that for almost every $y$, there is a unique optimal coupling between the $\mu^i(y)$ and that coupling is concentrated on the graph of a function over $\mu^0(y)$.  Furthermore, \cite{McCannPass20} implies that for almost every $y$ the conditional probabilities $\mu_i^y$ are concentrated on $m-n$ dimensional affine subspaces $X_1(y,Dv_i(y))=\{x:  x\cdot D f(y) =Dv_i(y)\}$ where $f:Y \rightarrow X$ locally parameterizes the submanifold supporting $\nu$, as in Remark \ref{rem: lower dimensional convention} and $v_i$ is the Kantorovich potential and, by  Theorem \ref{thm: structure of multi-to one-dimensional transport}, each $\mu_i^y$ is absolutely continuous with respect to $m-n$ dimensional Hausdorff measure. %\marginpar{BP: maybe the conditional probabilities and their densities is worth a Lemma.}
	%	in fact, these affine subspaces are the  level sets of $x \mapsto D_y[x\cdot f(y)] = x\cdot D f(y)$ given by $, 
	The subspaces  $X_1(y,Dv_0(y))$ and  $X_1(y,Dv_1(y))$ containing $T^{-1}_0(y)$ and $T^{-1}_1(y)$ respectively, are  parallel.   We decompose each $x = x^N +x^{\perp}$, where $x^N$ is in the $m-n$ dimensional null space $null(Df(y))$ of $Df(y)$ and $x^{\perp}$ in its $n$-dimensional orthogonal complement, $null(Df(y))^{\perp}$, on which the linear map $Df(y)$ is injective.  Then for $x_i \in X_1(y,Dv_i(y))$, $x_i^{\perp}$ must be the unique element of  $null(Df(y))^{\perp}$ such that $(x_i^{\perp})^T Df(y) =Dv(y_i)$; that is, $x_i^{\perp}$ is \emph{constant} on $X_1(y,Dv_i(y))$.
		
		The quadratic cost between  $X_1(y,Dv_0(y))$ and  $X_1(y,Dv_1(y))$ then takes the form $|x_0-x_1|^2 = |x_0^N-x_1^N|^2 +C$, where $C=|x_0^{\perp}-x_1^{\perp}|^2$ is constant throughout $X_1(y,Dv_0(y)) \times X_1(y,Dv_1(y))$ and therefore doesn't affect optimal transport.  The optimal transport problem between  $\mu_0^y$ and  $\mu_1^y$ is then equivalent to optimal transport with respect to the quadratic cost on  $null(Df(y))$, or equivalently,  $\mathbb{R}^{m-n}$.  The existence, uniqueness and structure (as the gradient of a convex potential) of the optimal mapping between them then follows from another application of Brenier's theorem.  Thus, for almost every $y$, the subset where the optimal map fails to be uniquely defined has $m-n$ Lebesgue measure $0$, yielding the desired result.
\end{proof}
We now turn to the study of minimizing $W_\nu$ geodesics. 
 \begin{proposition}\label{prop: geodesics}
		Suppose that $\pi_0, \pi_1 \in \Pi_{opt}(\nu)$.  Then $\pi_t = ((1-t)e_0 +te_1, e_2)_\#\gamma$ is a minimizing geodesic for the $\nu$-based Wasserstein metric on couplings defined by \eqref{eqn: metric on couplings}, for each optimal $\gamma$  in \eqref{eqn: metric on couplings}, where $e_0(x_0,x_1,y) =x_0$, $e_1(x_0,x_1,y) =x_1$, and $e_2(x_0,x_1,y) =y$,
		
	Suppose that $\mu_0,\mu_1 \in \mathcal{P}_\nu^u$.  Then the generalized geodesic  $\mu_t=((1-t)e_0 +te_1)_\#\gamma$ is a minimizing geodesic for the $\nu$ based Wasserstein metric for each optimal $\gamma$ in \eqref{eqn: metric definition}, provided that $\mu_t \in \mathcal{P}_\nu^u$ for all $t \in [0,1]$.
	
\end{proposition}
\begin{proof}
	We tackle the first statement first.  As in the proof of Lemma 9.2.1 in \cite{AmbrosioGradientFlows2008},  $\pi_t \in \Pi_{opt}(\nu)$.   It is clear that $\gamma_t =(e_0,(1-t)e_0 +te_1, e_2)_\#\gamma$ has two-fold marginals $\gamma_{x_0y} =\pi_0$ and $\gamma_{x_ty} =\pi_t$.  Therefore,
	
\begin{eqnarray*}
\tilde  W^2_\nu(\pi_0,\pi_t) &\leq& \int_{X \times X \times X}|x_0-x_t|^2d\gamma_t(x_0,x_t,y)\\
&=&\int_{X \times X \times X}|x_0-[(1-t)x_0 +tx_1]|^2d\gamma(x_0,x_1,y)\\
&=&\int_{X \times X \times X}t^2|x_0 -x_1|^2d\gamma(x_0,x_1,y) = t^2  \tilde W^2_\nu(\pi_0,\pi_1)
\end{eqnarray*}
So $ \tilde W_\nu(\pi_0,\pi_t) \leq t \tilde W_\nu(\pi_0,\pi_1) $.  A very similar argument implies  $ \tilde W_\nu(\pi_t,\pi_s) \leq (s-t) \tilde W_\nu(\pi_0,\pi_1) $ for $s>t$, and  $ \tilde W_\nu(\pi_s,\pi_1) \leq (1-s) \tilde W_\nu(\pi_0,\pi_1) $.  Adding these three inequalities yields
\begin{eqnarray*}
 \tilde W_\nu(\pi_0,\pi_t)+  \tilde W_\nu(\pi_t,\pi_s) +W_\nu(\pi_s,\pi_1)&\leq& t  \tilde W_\nu(\pi_0,\pi_1) + (s-t) \tilde W_\nu(\pi_0,\pi_1)+(1-s) \tilde W_\nu(\pi_0,\pi_1)\\
 &=&  \tilde W_\nu(\pi_0,\pi_1)
\end{eqnarray*}
Since two applications of the triangle inequality imply the opposite inequality, $ \tilde W_\nu(\pi_0,\pi_1) \leq  \tilde W_\nu(\pi_0,\pi_t)+  \tilde W_\nu(\pi_t,\pi_s) +W_\nu(\pi_s,\pi_1)$, we must have equality throughout; in particular 
   $ \tilde W_\nu(\pi_t,\pi_s) = (s-t) \tilde W_\nu(\pi_0,\pi_1) $, and so $\pi_t$ is a minimizing geodesic for the $\nu$-based Wasserstein metric on couplings.

The proof of the second assertion is very similar and is omitted.  The only subtely is that in general it is not clear to us that $\mu_t$ remains in the set $\mathcal{P}_\nu^u$ on which $W_\nu$ is a metric, so that property is added as an assumption.
\end{proof}
Under much stronger conditions on the marginals $\mu_0$ and $\mu_1$ and the reference measure $\nu$, we are able to use Theorem \ref{thm: graphical solutions} to show that the geodesic for the $W_\nu$ metric between $\mu_0$ and $\mu_1$ is unique.

\begin{lemma}[Structure of geodesics]\label{lem: structure of geodesics}
	Let $\mu_0$, $\mu_1$ and $\nu$ be as in Theorem \ref{thm: structure of multi-to one-dimensional transport}.   Then the geodesic $\mu_t$ between $\mu_0$ and $\mu_1$ is uniquely defined.  The Kantorovich potential for optimal transport between $\mu_t$ and $\nu$ is $v_t =tv_1 +(1-t)v_0$, where $v_i:Y \rightarrow \mathbb{R}$ is the Kantorovich potential between $\mu_i$ and $\nu$, for $i=0,1$.  For almost every $y$, the conditional probability $\mu_t^y$, concentrated on $X_1(y,Dv_t(y))$, is the unique displacement interpolant between $\mu^y_0$ and $\mu_1^y$.  It is absolutely continuous with respect to $(m-n)$-dimensional Hausdorff measure; denoting its density by $\bar \mu_t^y$, and the density of $\mu_t$ with respect to $\mathcal{L}^m$ by $\bar \mu_t$, we have, $\bar \mu_t^y(x_t) =\bar \mu_t(x_t)/(JT^y_t(x_t)\bar \nu(y))$ where $JT^y_t(x_t):=\frac{\sqrt{\det(Df^TDf(y))}}{\det[-D^2v_t(y)-x_t \cdot D^2f(y)]}$ is the Jacobian of the optimal map $T_t$ from $\mu_t$ to $\nu$, evaluated at $x_t \in T_t^{-1}(y) \subseteq X_1(y,Dv_t(y))$.

		The Brenier map between $\mu_0^y$ and $\mu_t^y$ is then given by  $\nabla \phi^y_t=(1-t)I +t\nabla \phi^y$, where $\nabla \phi^y: T_0^{-1}(y) \rightarrow T_1^{-1}(y)$ is the Brenier map between $\mu_0^y$ and $\mu_1^y$.
	
\end{lemma}
\begin{proof}
	Constructing $v_t$ and each $\mu_t^y$ as in the statement of the Lemma, it is clear from Theorem \ref{thm: structure of multi-to one-dimensional transport} that $v_t$ is the Kantorovich potential for optimal transport between $\mu_t(x) = \int_Y\mu^y_t(x)d\nu(y)$ and $\nu$, and that $\pi_t=\nu(y)\otimes\mu_t^y(x)$ is the unique optimal transport plan between $\mu_t$ and $\nu$. The first part of  Theorem \ref{thm: charaterization of metric} then implies that, for any $s,t$
\begin{equation}\label{eqn: w nu geodesic}
W^2_\nu(\mu_s,\mu_t) =\int_XW_2^2(\mu_t^y, \mu_s^y)d\nu(y) = (t-s)^2\int_XW_2^2(\mu_0^y, \mu_1^y)d\nu(y)=(t-s)^2W^2_\nu(\mu_0,\mu_1)
\end{equation}
since each $\mu_t^y$ is a displacement interpolant (that is, a Wasserstein geodesic) between $\mu_0^y$ and $\mu_1^y$. This implies that $\mu_t$ is a $W_\nu$ geodesic.

  %Similarly, 
%\begin{equation}
%W^2_\nu(\mu_1,\mu_t) =\int_XW_2^2(\mu_t^y, \mu_1^y)d\nu(y) = (1-t)^2\int_XW_2^2(\mu_1^y, \mu_0^y)d\nu(y).
%$$	 
Furthermore, the uniqueness of the displacement interpolants $\mu_t^y$ together implies that  any other curve $\tilde \mu_t \in \mathcal P^u_\nu(X)$, disintegrated with respect to $\nu$ as $\tilde \mu_t^y(x)\otimes \nu(y)$, cannot satisfy \eqref{eqn: w nu geodesic} for all $s,t$ (as otherwise $\tilde \mu_t^y$ would be a Wasserstein geodesic between $\mu^y_0$ and $\mu^y_1$ for $\nu$ all $y$, which would violate uniqueness). This yields the desired uniqueness of the $W_\nu$ geodesic.

The absolute continuity of $\mu^y_t$ with respect to $(m-n)$-dimensional Hausdorff measure on the affine set $X_1(y,Dv_t(y)$ follows from standard arguments, since it is the displacement interpolant of two absolutely continuous measures on parallel affine sets.   Absolute continuity of $\mu_t$ and the formula relating $\bar \mu_t$ and $\bar \mu_t^y$ then follows from Theorem \ref{thm: structure of multi-to one-dimensional transport}.
 Finally, the asserted structure $\nabla \phi^y_t=(1-t)I +t\nabla \phi^y$ of the Brenier map between between $\mu_0^y$ and $\mu_t^y$ then follows by standard optimal transport arguments applied to the parallel affine subspaces $X_1(y,Dv_0(y))$ and  $X_1(y,Dv_1(y))$.

\end{proof}

\section{Convexity properties}\label{sect: convexity}
We now turn our attention to  certain convexity properties of $W_\nu$ geodesics.  We begin by studying the convexity of various functionals along the geodesics, and then show that the set of measures satisfying a slight strengthening of the generalized nested condition is geodesically convex.

\subsection{Geodesic convexity of functionals}

We consider the functionals
\begin{eqnarray}
\mu& \mapsto& W_2^2(\mu,\nu) \label{eqn: wass distance}\\
\mu &\mapsto& \int_X V(x)d\mu(x) \label{eqn: potential energy}\\
\mu &\mapsto &\int_{X^2} W(x-z)d\mu(x)d\mu(z)\label{eqn: interaction energy}
\end{eqnarray}
and 
\begin{equation}
\label{eqn: internal energy}
\mu \mapsto \left\{
  \begin{array}{l l}
    \int U(\bar \mu(x) )dx &\quad \text{ if } d\mu =\bar \mu(x)dx \text{ is a.c. wrt Lebesgue measure on }X\\
    +\infty & \quad \text{ otherwise.}
  \end{array}\right.
\end{equation}
Convexity of each of these functionals along more classical interpolants has been studied extensively.  Under suitable conditions, specified below, functionals \eqref{eqn: potential energy} -- \eqref{eqn: internal energy} were first shown to be convex along Wasserstein geodesics (that is, displacement convex) by McCann \cite{mccann1997convexity}.  Functionals \eqref{eqn: wass distance} -- \eqref{eqn: interaction energy} were shown to be convex along \emph{every} generalized geodesic with base measure $\nu$ in \cite{AmbrosioGradientFlows2008}.  Functional \eqref{eqn: internal energy}, on the other hand, is not convex along every generalized geodesics, but it is shown in Proposition 9.3.1 of \cite{AmbrosioGradientFlows2008} that for each $\mu_0,\mu_1 \in \mathcal{P}(X)$, there exists a generalized geodesic joining $\mu_0$ to $\mu_1$ along which it is convex.  In particular,  when the base measure $\nu$ is absolutely continuous with respect to Lebesgue measure, and so the generalized geodesic between any $\mu_0$ and $\mu_1$ is unique, \eqref{eqn: internal energy} is convex along every generalized geodesics, but this is often not true for singular base measures.

 We will show that these functionals are geodesically convex  with respect to the metric $W_\nu$ (under the usual conditions, specified below). Explicitly, geodesic convexity means that $\mathcal F(\mu_t)$ is a convex function of $t \in [0,1]$, where $\mathcal F:\mathcal P(X) \rightarrow \mathbb{R}$ is any of \eqref{eqn: interaction energy}, \eqref{eqn: internal energy}, \eqref{eqn: potential energy} or \eqref{eqn: wass distance}.  We will restrict our attention to $W_\nu$ geodesics $\mu_t$ of the form in Proposition \ref{prop: geodesics}; that is, $\mu_t =(te_0+(1-t)e_1)_\#\gamma$, where $\gamma$ is optimal in \eqref{eqn: metric definition}, under the assumption that each $\mu_t \in \mathcal{P}^u_\nu(X)$, since in more general situations the existence of a minimizing $W_\nu$ geodesic in the metric space $\mathcal{P}_\nu^u$ joining $\mu_0$ and $\mu_1$ is not clear.  We are particularly interested in the case when the endpoints $\mu_0$ and $\mu_1$ and reference measure $\nu$ satisfy the assumptions in Theorem \ref{thm: graphical solutions}, under which the existence of a unique geodesic follows from Lemma \ref{lem: structure of geodesics}. % and the geodesic $\mu_t$ is given by $[tG+(1-t)Id]_\# \mu_0$, where $G$ is the  optimal mapping (for our metric) between $\mu_0$ and $\mu_1$, guaranteed to exist by Theorem \ref{thm: graphical solutions}.   

	Since the geodesics we consider  for $W_\nu$ are always generalized geodesics, the following result follows immediately from Lemma 9.2.1 and Propositions 9.3.2 and 9.3.5 \cite{AmbrosioGradientFlows2008}.
\begin{proposition}
	The following functionals are  convex along all minimizing $W_\nu$ geodesics of the form in Proposition \ref{prop: geodesics}, under the corresponding conditions:
	\begin{enumerate}
		\item \eqref{eqn: wass distance} is geodesically  1-convex for any $\nu$.
		\item \eqref{eqn: potential energy} is geodesically convex if $V$ is convex.  It is geodesically strictly convex if $V$ is strictly convex.
		\item \eqref{eqn: interaction energy} is geodesically convex if W is convex.  It is geodesically strictly convex along the subset of measures with the same barycenter  if W is strictly convex.
	\end{enumerate}
\end{proposition}%\marginpar{BP: In fact, these are known to be convex from the Ambrosio-Gigli-Savare formulation of generalized geodesics (our interpolation is a special case).  For the internal energy, AGS proves convexity for SOME generalized geodesic between any $\mu_0$ and $\mu_1$ -- our work gives an explicit construction of one.}
% under the usual conditions on $U$, provided $\nu$ is supported on a one dimensional curve.
The geodesic convexity of the internal energy is somewhat more involved; in particular, as discussed above, unlike the other three forms, it does not hold for all generalized geodesics. %\footnote{\color{blue} Note that Proposition 9.3.9 in \cite{AmbrosioGradientFlows2008} essentially asserts that for any given base measure $\nu$ and any $\mu_0$ and $\mu_1$, there \emph{exists} a generalized geodesic $t\mapsto \mu_t$ along which the internal energy is convex.  However, when the generalized geodesic between $\mu_0$ and $\mu_1$ is non-unique, there will often be generalized geodesics along which convexity fails.  We show below (Corollary \ref{cor: internal energy convexity}), under certain conditions,  that convexity holds along the $\nu$-based Wasserstein geodesic, a particular generalized geodesic.}. 
 Our proof, under additional restrictions on $\nu$, uses the structure of the transport map $G$ in a crucial way.

We assume that the reference measure $\nu$ is absolutely continuous with respect to $n$-dimensional Lebesgue measure on the $n$-dimensional submanifold $Y$ (using the convention in Remark \ref{rem: lower dimensional convention}) and consider only geodesics $\mu_t$ whose endpoints $\mu_0$ and $\mu_1$ are absolutely continuous with respect to $m$-dimensional Lebesgue measure on $X$ with densities with respective densities $\bar \mu_0$ and  $\bar \mu_1$. Letting $G$ be the optimal map in \eqref{eqn: metric definition}, guaranteed to exist by Theorem \ref{thm: graphical solutions}, we will use $x_t =(1-t)x_0+tG(x_0)$ to denote a point in the support of $\mu_t$. Lemma \ref{lem: structure of geodesics} then yields uniqueness of the geodesic $\mu_t$, and we adopt similar notation to Theorem \ref{thm: structure of multi-to one-dimensional transport} and Lemma \ref{lem: structure of geodesics}: $v_i$ is the Kantorovich potential for transport between $\mu_i$ and $\nu$, and $v_t =(1-t)v_0+tv_i$ plays the same role between $\mu_t$ and $\nu$.  Letting $T_t$ be the optimal transport map sending $\mu_t$ to $\nu$, set  $JT^y_t(x_t):=\frac{\sqrt{\det(Df^TDf(y))}}{\det[-D^2v_t(y)-x_t \cdot D^2f(y)]}$, and $\bar \mu_t^y(x_t) =\bar \mu_t(x_t)/(JT^y_t(x_t)\bar \nu(y))$ the density with respect to $(m-n)$-dimensional Hausdorff measure of the conditional probability $\mu_t^y$ after disintegrating $\mu_t$ with respect to $\nu$, using the map $T_t$; that is, $\mu_t(x_t) :=\mu_t^y(x_t)\otimes \nu(y)$.

We recall that  Lemma \ref{lem: structure of geodesics} imlies that  $\mu_t^y$ is the displacement interpolant between $\mu_0$ and $\mu_1$ and that the Brenier map between $\mu_0^y$ and $\mu_t^y$ is then given by $\nabla \phi^y_t=(1-t)I +t\nabla \phi^y$, where $\nabla \phi^y: T_0^{-1}(y) \rightarrow T_1^{-1}(y)$ is the Brenier map between $\mu_0^y$ and $\mu_1^y$.

\begin{lemma}\label{lem: formula for internal energy of interpolant}Assume that the reference measure $\nu$ is absolutely continuous with respect to $n$-dimensional Hausdorff measure on a smooth $n$-dimensional submanifold $Y$ and that $\mu_0$ and $\mu_1$ are absolutely continuous with respect to Lebesgue measure on $X$.
	The internal energy at the interpolant is given by
	\begin{equation}\label{eqn: int energy change of variables}
	\begin{split}
		&\int_XU(\bar \mu_t(x))dx=\\
		&\int_Y\frac{1}{\bar \nu(y)}\int_{T_0^{-1}(y)}U\Big(\frac{\bar \mu_0^y(x_0)\bar \nu(y)JT^y_t(\nabla\phi^y_t(x_0))}{\det(D^2\phi^y_t(x_0))}\Big)\frac{1}{JT^y_t(\nabla\phi^y_t(x_0))}\det(D^2\phi_t^y(x_0))d\mathcal H^{m-n}(x_0)dy 
		\end{split}
	\end{equation}

\end{lemma}
\begin{proof}
	Similarly to \eqref{eqn: change of variables}, we can write the internal energy as%\marginpar{BP: The coarea formula doesn't apply, since the optimal map may not be Lipschitz.  I think we can proceed as in my paper with McCann.}
	$$
	\int_XU(\bar \mu_t(x_t))dx_t =\int_Y\frac{1}{\bar \nu(y)}\int_{T_t^{-1}(y) \cap X(v_t^c)}U(\bar \mu_t^y(x_t)\bar \nu(y)JT^y_t(x_t))\frac{1}{JT^y_t(x_t)}d\mathcal H^{m-n}(x_t)dy.	
	$$
	The remainder of the proof is then a matter of rewriting the inner integral as an integral over $T_0^{-1}(y)$, which is standard.	 For fixed $y$, standard optimal transport theory applied to the parallel affine subspaces $X_1^c(y,Dv_i(y))$  implies that   the optimal map $\nabla \phi^y$ is differentiable $\mathcal{H}^{m-n}$ almost everywhere and the Monge-Ampere equation is satisfied for $\mathcal H^{m-n}$ almost every $x_0$:
	$$
	\bar \mu_0^y(x_0)=\bar \mu_1^y(\nabla \phi^y(x_0))\det D^2\phi^y(x_0)
	$$
	and similarly note that since $\nabla \phi^y_t = (1-t)I+t\nabla \phi^y: T_0^{-1}(y) \rightarrow T_t^{-1}(y)$ pushes $\mu_0^y$ forward to $\mu_y^t$ and is optimal, we have
	
	$$
	\bar \mu_0^y(x_0)=\bar \mu_t^y(\nabla \phi^y_t(x_0))\det D^2\phi^y_t(x_0)=\bar \mu_t^y(\nabla \phi^y_t(x_0))\det((1-t)I_{m-n}+t D^2\phi^y(x_0)).
	$$
	
	The result then follows by using the standard optimal transport change of variables formula on $X_1(y,Dv_0(y))$, using the change of variables $x_t=\nabla \phi^y_t(x_0) = (1-t)x_0+t\nabla \phi^y(x_0)$
	%	In addition, with $x_t = \nabla \phi^y_t(x_0)$ interpolating linearly between $x_0$ and $x_1 =\phi^y_t(x_0)$, and the potential $v_t =(1-t)v_0 +tv_1$, we have that
	%	\begin{eqnarray*}
	%	\int_XU(\bar \mu_t(x))dx_t =
	%	\end{eqnarray*}
\end{proof}
\begin{remark} It is natural to note that the factor $\det\Big(D^2\phi^y_t(x_0)\Big)/JT^y_t(\nabla\phi^y_t(x_0))$ appearing in \eqref{eqn:  int energy change of variables} is the determinant of the Jacobian matrix $DG_t(x_0)$ for the optimal mapping $G_t$ in Theorem \ref{thm: graphical solutions} between $\mu_0$ and $\mu_t$.  It may be helpful to note that the block diagonal matrix $M_t$ introduced in the proof of the following result  is \emph{not} equal to $DG_t(x_0)$.  $DG_t(x_0)$ is in fact block upper triangular, with the same diagonal blocks are $M_t$.  The key point in the proof is the observation that $DG_t(x_0)$ and $M_t$ have the same determinant;  it is then more convenient to think of $\det\Big(D^2\phi^y_t(x_0)\Big)/JT^y_t(\nabla\phi^y_t(x_0))$ as the determinant of $M_t$, since this matrix is positive definite and we can apply Minkowski's determinant inequality to it. 
\end{remark}
\begin{corollary}\label{cor: internal energy convexity}
	Under the assumption in the preceding Lemma, the internal energy \eqref{eqn: internal energy} is geodesically convex assuming $U(0)=0$ and $r\mapsto r^mU(r^{-m})$ is convex and non-increasing.
\end{corollary}
\begin{proof}
	Note that $\det\Big(D^2\phi^y_t(x_0)\Big)/JT^y_t(\nabla\phi^y_t(x_0))$ is in fact the determinant of the $m \times m$ block diagonal matrix $M_t$ with an upper $(m-n) \times (m-n)$ block $D^2\phi^y_t(x_0)=(1-t)I_{m-n} +tD^2\phi^y(x_0)$ and a lower $n\times n$ block $\frac{1}{|\det(Df^TDf(y))|^{1/2n}}[-D^2v_t(y)-x_t \cdot D^2f(y)]$ whose determinant is $\frac{1}{JT^y_t(\nabla\phi^y_t(x_0)}=\frac{1}{|\det(Df^TDf(y))|^{1/2}}\det[-D^2v_t(y)-x_t \cdot D^2f(y)]$. Noting that  $x_t = \nabla \phi^y_t(x_0)$ interpolates linearly between $x_0$ and $x_1 =\nabla\phi^y_t(x_0)$, and the potential $v_t =(1-t)v_0 +tv_1$ for optimal transport between $\mu_t$ and $\nu$ interpolates linearly between $v_0$ and $v_1$, $M_t$ is affine in $t$.
	
	 Now, the first block $D^2\phi^y_t(x_0)$ of $M_t$ is clearly non-negative definite by convexity of $\phi^y$.   Note that since $v_t$ is the Kantorovich potential for transport between $\mu_t$ and $\nu$, we have for $y=T_t(x_t)$, $D^2v_t(y) \leq D^2_{yy}c(x_t,y) = x_t\cdot D^2f(y)$; therefore, the second block $\frac{1}{{(\det(Df^TDf(y)))^{1/{2n}}}}[-D^2v_t(y)-x_t \cdot D^2f(y)]$ of $M_t$ is also non-negative definite. It follows that
%	Since both blocks are diagonal non-negative definite, 
$M_t$ is as well, and Minkowski's determinant inequality implies that $(\det M_t)^{1/m}$ is concave as a function of $t$; the convexity of the integrand now follows by a standard argument.  Indeed, the integrand is the composition of the concave mapping $t \mapsto [\det((1-t)I_{m-n} +tD^2\phi^y(x_0))/JT^y_t(\nabla\phi^y_t(x_0))]^{1/m} =(\det[M_t])^{1/m}$ and the convex, non-increasing mapping $r \mapsto U(\frac{\bar \mu_0^y(x_0)\bar \nu(y)}{r^m})r^m$ and therefore convex.
\end{proof}
\begin{remark}An important special case occurs when the function $f$ is affine, in which case $c(x,y) =-x\cdot f(y)$ becomes an index cost.  As shown in Proposition \ref{prop: layerwise Wasserstein}, when $Y \subseteq \mathbb{R}$ is one dimensional, the metric $W_\nu$ becomes the  layerwise Wasserstein metric from \cite{KimPassSchneider19}, and Corollary \ref{cor: internal energy convexity} generalizes Corollary 4.2 there (which asserts convexity of the Shannon entropy along layerwise-Wasserstein geodesics).
\end{remark}

%\marginpar{BP: This gives uniqueness of minimizers of \eqref{eqn: interaction energy}+\eqref{eqn: internal energy}+\eqref{eqn: potential energy}+\eqref{eqn: wass distance} as a Corollary, but this already follows from the Ambrosio-Gigli-Savare theory.  But we get convexity of this functional with respect to a well defined metric, which might be useful for gradient flow or steepest descent schemes.}

\begin{remark}%\marginpar{BP: Is this really true? It seems that the $x_tD^2f(y)$ is not constant.}
In the case where $f$ is affine, so that $Df(y)$ is constant, slightly improved convexity results can be obtained when we restrict our interpolations to the class $\mathcal P_{\nu,v}(X)$ of probability measures on $X$ whose Kantorovich potential corresponding to optimal transport to $\nu$ is a fixed $v$ (up to addition of a constant). (It is worth noting that in our context, although the potential $v$ does not uniquely determine the measure $\mu$, which would make $\mathcal P_{\nu,v}(X)$  a singleton, it is easy to see that $\mathcal P_{\nu,v}(X)$ is geodesically convex).

 In this case, since $v_t=v$ is constant, $X_1(y, Dv(y))=\{x:  x\cdot D f(y) =Dv(y)\}$ on which $\mu^y_t$ is supported is constant as well.  The curve $\nabla \phi_t^y(x_0) =(1-t)x_0+t\nabla \phi^y(x_0)$ then moves along $X_1(y, Dv(y))$, so that 
 $JT^y_t(\nabla\phi^y_t(x_0))$ is constant as a function of $t$, and so we can use concavity of $\det^{m-n}$ on the set of positive definite $(m-n) \times (m-n)$ dimensional matrices to obtain convexity of  \eqref{eqn: internal energy} as soon at $r\mapsto r^{m-n}U(r^{n-m})$ is convex and nonincreasing, a slight improvement over the classical condition.  

The set $\mathcal P_{\nu,v}(X)$ is potentially relevant in the root comparison applications in  \cite{KimPassSchneider19}.  In these applications, the main objective is to interpolate between the shapes, and as part of the analysis they are first rescaled so that they have uniform marginals in one (the vertical) direction; that is, one compares measures $\tilde \mu_0$ and $\tilde \mu_1$ on $[0,1] \times \mathbb{R}^{m-1}$, whose first marginals are Lebesgue.  In this case, if we set $\nu$ to be Lebesgue measure on the line segment $[0,1] \times\{0\}$, then optimal transport between $\mu_i$ and $\nu$ is given by the projection map $x \mapsto x^1$ and the Kantorovich potential on the $y$ side are $v_i(y)=y^2/2$ for both $i=0,1$.

Therefore, interpolating between these rescaled versions corresponds to working in  $\mathcal P_{\nu,v}(X)$, where $\nu$ is uniform measure on $\{0\} \times [0,1]$ and $v(y)=y^2/2$, and the slightly improved convexity results described above apply.
\end{remark}

\subsection{Geodesic convexity of the set of nested sources}
Assuming the $\nu$ concentrates on a smooth $n$-dimensional submanifold $Y$, we now turn our attention to the structure of the set of measures $\mu \in \mathcal P^u_\mu(X)$ such that optimal transport with quadratic cost between $\mu$ and $\nu$ satisfies (a strengthening of) the generalized nested condition.  We will show that this set is $W_\nu$ convex.

We adopt the convention in Remark \ref{rem: lower dimensional convention}; the cost is $c(x,y) =-x\cdot f(y)$, where $f:Y \rightarrow \mathbb{R}^m$ for some $Y \subset \mathbb{R}^n$. Note that this cost naturally extends to all $x \in \mathbb{R}^m$ using the same formula.  Our strengthening of the nestedness condition requries some additional notation.  For each $y\in Y$, decompose $x \in \mathbb{R}^m$ into $x=x^y_{||}+x^y_\perp$ with $x^y_{\perp}\cdot Df(y)=0$  and $x^y_{||}$ orthogonal to the null space of $Df(y)$.

We define the \emph{enhancement of $X$ by $y \in Y$} as

$$
X^y=\{x=x^y_{\perp} +x^y_{||}:x^y_\perp +\tilde x^y_{||}, \tilde x^y_\perp + x^y_{||} \in X \text{ for some } \tilde x^y_{||} \in [null(Df(y))]^\perp,\tilde x^y_{\perp} \in null(Df(y)) \}
%X^E=\cup_{y \in Y}\{x=x^y_{\perp} +x^y_{||}:x^y_\perp +\tilde x^y_{||}, \tilde x^y_\perp + x^y_{||} \in X \text{ for some } \tilde x^y_{||},\tilde x^y_{\perp}\}
$$
We say that  $(c,\mu,\nu)$ satisfies the \emph{enhanced generalized nestedness condition}  at $y$ if 
\[
X_1^y(y, Dv(y)) = \partial^cv(y):=\{x\in X^y: -x\cdot f(y)-v(y) \leq -x\cdot f(\tilde y)-v(\tilde y)\; \forall \tilde y \in Y \},
\]
where $v$ is the Kantorovich potential.
We say that $(c,\mu,\nu)$ satisfies the \emph{enhanced generalized nestedness condition} if it satisfies this condition at almost every $ y\in Y$.  Since $X \subseteq X^y$ for each $y$, $X_1(y, Dv(y)) \subseteq X_1^y(y, Dv(y))$, and so the enhanced generalized nestedness condition implies the generalized nestedness condition $X_1(y, Dv(y)) = \partial^cv(y)$ introduced in \cite{McCannPass20}.

\begin{proposition}
	%	Let $c$ be the quadratic cost between a fixed measure in $\mathbb{R}^m$ and a variable measure on an $n$-dimensional submanifold.  Then 
	Let $\nu$ be absolutely continuous with respect to $n$-dimensional Hausdorff measure on the $n$-dimensional submanifold $Y$.  Then the set of $\mu$ such that  $(c, \mu,\nu)$ is enhanced generalized nested at $y$ is geodesically convex with respect to the metric $W_\nu$.
	
	%For fixed $\mu$ on $X$, the set of $\nu$ such that $(c, \mu,\nu)$ is enhanced generalized nested at $y$ is geodesically convex with respect to the metric $(D^*)^2_\mu$.  
	%	On the other hand, if $\nu$ is fixed, the set of $\mu$ such that  $(c, \mu,\nu)$ is strongly generalized nested is geodesically convex.
	
\end{proposition}
\begin{proof}
 
	%Consider the first claim.
	Given Kantorovich potentials $v_0$ and $v_1$ corresponding to $\mu_0$ and $\mu_1$, the potential $v_t:=(1-t)v_0+tv_1$ corresponds to the interpolant $\mu_t$.  
	Observe that, since $D_yc(x,y) =-x\cdot Df(y)$ is affine in $x$, the enhanced level sets $X_1^y(y, Dv_t(y)) = (1-t)X_1^y(y, Dv_0(y))+tX_1^y(y, Dv_1(y))$ is the Minkowski sum of the two endpoint enhanced level sets.  Letting $x_t =(1-t)x_0 +tx_1 \in X_1^y(y, Dv_t(y))$, with each $x_i \in X_1^y(y, Dv_i(y))$, we have by assumption that
	$$
	 -x_i\cdot f(y)-v_i(y) \leq x_i\cdot f(\tilde y)-v_i(\tilde y)
	$$
	for all $ \tilde y\in Y$.
	Taking the weighted average of these two inequalities yields
	\begin{eqnarray*}
	-(1-t)x_0\cdot f(y)-(1-t)v_0(y) -tx_1\cdot f(y)-tv_1(y)&\leq& -(1-t)x_0\cdot f(\tilde y)-(1-t)v_i(\tilde y)\\
	&&-tx_1\cdot f(\tilde y)-tv_1(\tilde y)\\
	-x_t\cdot f(y)-v_t(y) &\leq& x_t\cdot f(\tilde y)-v_t(\tilde y)
	\end{eqnarray*}
	for all $\tilde y \in Y$.  Thus, $x_t \in \partial^cv_t(y)$ as desired. 
	
	%The proof of the second claim is almost identical, since for fixed $\mu$ on $X$ the Kantorovich potential $v_t$ corresponding to the $D_\mu^*$-geodesic $\nu_t$ between $\nu_0$ and $\nu_1$ is also affine, $v_t= (1-t)v_0 +tv_1$.  The argument above then applies.

%	Note that the level set $X^E_{= }(y,v_t'(y))$ is exactly the Minkowski interpolant, $(1-t)X^E_{= }(y,v_0'(y))+tX^E_{= }(y,v_t'(y))$.  Therefore, given $x \in X^E_{= }(y,v_t'(t))$, we have $x = (1-t)x_0 +tx_1$ for $x_i \in X^E_{= }(y,v_i'(y))$, and so
%	\begin{eqnarray*}
%		D^2v_t(y) &=&(1-t)D^2v_0(t) + tD^2v_1(t) \\
%		&< &(1-t)x_0\cdot D^2f(y)+tx_1\cdot D^2f(y)\\
%		&=&x_tD^2f(y)
%	\end{eqnarray*}
\end{proof}

\section{Fixed point characterization of solutions to variational problems}\label{sect: fixed point}
In this section, we consider a fixed measure $\mu$ on the bounded domain $X \subseteq \mathbb{R}^m$, absolutely continuous with respect to Lebesgue measure with density $\bar \mu$, bounded above and below, and a fixed domain $Y:=(\underline y, \overline y) \subset \mathbb{R}$.  We are interested in the minimization problem \eqref{eqn: variational problem}.
 We recall that this problem belongs to a class arising in certain game theory problems, introduced in a series of papers by Blanchet-Carlier \cite{blanchet2014remarks, abgcmor,abgcptrl}, in which the fixed measure $\mu$ represents a distribution of players and $Y$ a space of strategies; minimizing measures $\nu$ in \eqref{eqn: variational problem} represent equilibrium distributions of strategies.    In general, the dimensions of $m$ and $n$ of the spaces of players and strategies represent the number of charateristics used to distinguish between them.  The case $m >n$ (including the case $n=1$ treated here) are of particular interest, since it is often the case that players are highly homogeneous, whereas the strategies available to them are less varied. 
	
	When $m=1$ (so that both players and strategy spaces are one dimensional), Blanchet-Carlier introduced a characterization of minimizers as fixed points of a certain mapping.  They then iterated that mapping as a way to compute solutions \cite{blanchet2014remarks}; however, they did not prove convergence of this scheme.  
	
	Here we introduce a similar fixed point characterization for $m \geq 1$.  Furthermore, we prove for any $m$, under certain assumptions, that the mapping is a contraction with respect to a strategically chosen metric, implying convergence of the iterative scheme; even for $m=1$, this is a new result.  The key to the proof is choosing a suitable metric; our choice is a slight variant of the dual metric $W^*_{\mu,\mathcal{L},c,1}$ introduced in Section \ref{sect:dual metrics}.

%\begin{equation}\label{eqn: variational problem}
%	\nu \mapsto T_c(\mu,\nu) + \int \bar \nu(y) \log(\bar \nu(y))dy,
%	\end{equation}

%among probability measures in $P([\underline y, \overline y])$.    Here $T_c(\mu,\nu):\inf_{\pi \in \Pi(\mu,\nu}\int_{X \times Y}c(x,y)d\pi(x,y)$ is the cost of optimal transportation between $\mu$ and $\nu$. 
%We introduce below a fixed point characterization of solutions.  A similar characterization was introduced by Blanchet-Carlier when $m=1$, and the map was iterated as a way to approximate solutions \cite{blanchet2014remarks}; however, they did not prove convergence of this scheme.  Here, we prove for any $m$, under certain assumptions, that the mapping is a contraction with respect to a strategically chosen metric, implying convergence of the iterative scheme.  %The proof implicitly uses ideas related to the dual metric \eqref{eqn: dual metric}, however, the actual metric yielding a contraction differs slightly from \eqref{eqn: dual metric} for two reasons: we allow more general cost functions $c$ (\eqref{eqn: dual metric} was defined for potentials arising from the quadratic cost), and, instead of using Kantorovich potentials from optimal transport, we use functions defined from a population splitting equation as in \cite{PassM2one} (which coincide with derivative of Kantorovich potentials for nested models but not in general).

Choose $\underline d =\min_{x \in \bar X,y\in \bar Y}\frac{\partial c}{\partial y}(x,y), \overline d =\max_{x \in \bar X,y\in \bar Y}\frac{\partial c}{\partial y}(x,y) $. %and set $d =\max\{\overline d, -\underline d\}$

Let $k \in L^1\Big((\underline y, \overline y); [\underline d, \overline d]\Big)$, the set of $L^1$ functions on $(\underline y, \overline y)$ taking values in $[\underline d, \overline d]$, and set $v_k(y): =\int_{\underline y}^{y}k(s)ds$.  Note that  $L^1\Big((\underline y, \overline y); [\underline d, \overline d]\Big)$ is a complete metric space. Let

\begin{equation} \label{eqn: mapping potential to measure}
\bar \nu_k(y):=\frac{e^{-v_k(y)}}{\int_{\underline y}^{\overline y}e^{-v_k(s)}ds}.
\end{equation}

This is a probability density on $(\underline y, \overline y)$; we denote the corresponding measure by $\nu_k(y)$.   Note that \eqref{eqn: mapping potential to measure} then defines a mapping, $L^1\Big((\underline y, \overline y); [\underline d, \overline d]\Big) \rightarrow \mathcal P^{ac}((\underline y, \overline y))$, where  $\mathcal P^{ac}((\underline y, \overline y))$ is the set of absolutely continuous probability measures on  $(\underline y, \overline y)$:
\begin{equation}\label{eqn: mapping potential to measure 2}
	k \mapsto \nu_k
\end{equation}

We note that, for any $k\in L^1\Big((\underline y, \overline y); [\underline d, \overline d]\Big)$, we have
\begin{eqnarray}\label{eqn: lower bound on nu}
	\nu_k(\underline{y},y) &=&\int_{\underline y}^{y} \bar \nu_k(s)ds\nonumber\\
	&\geq& \int_{\underline y}^{y}\frac{\underline{d}e^{-\overline{d}(s-\underline{y})}}{1-e^{-\underline{d}(\overline{y}-\underline{y})}}ds\nonumber\\
	&=&\frac{\underline{d}(1-e^{-\overline{d}(y-\underline{y})})}{\overline{d}(1-e^{-\underline{d}(\overline{y}-\underline{y})})}.
\end{eqnarray}

Similarly,
\begin{equation}\label{eqn: lower bound on nu2}
\nu_k(y,\overline{y}) \geq \frac{\underline{d}(e^{-\overline{d}(y-\underline{y})}-e^{-\overline{d}(\overline{y}-\underline{y})})}{\overline{d}(1-e^{-\underline{d}(\overline{y}-\underline{y})})}.
\end{equation}

%$$
%\bar \nu_k(y) \geq \frac{e^{-\overline d (\overline y-\underline y)}}{(\overline y-\underline y )e^{-\underline d (\overline y-\underline y)}}
%$$
%\marginpar{BP: TO DO: This bound on $\bar \nu_k$ and the step in inequality (10) can be improved using local information.}

 We then note that, for a given $\nu \in \mathcal P^{ac}((\underline y, \overline y))$ the mass splitting property \eqref{eqn: mass splitting} selects a unique $k=k(y) \in [\underline d, \overline d]$ for each $y \in (\underline y, \overline y)$.  The resulting function $k(y)$ is then bounded and so clearly in $L^1\Big((\underline y, \overline y); [\underline d, \overline d]\Big)$.  We can then think of this process as defining a mapping $\mathcal P^{ac}((\underline y, \overline y)) \rightarrow L^1\Big((\underline y, \overline y); [\underline d, \overline d]\Big)$, $\nu \mapsto k(\cdot)$.  Composing this mapping with \eqref{eqn: mapping potential to measure 2} then yields a $F: L^1\Big((\underline y, \overline y); [\underline d, \overline d]\Big) \rightarrow L^1\Big((\underline y, \overline y); [\underline d, \overline d]\Big)$.  Explicitly, defining  $\tilde k:=F[k]$, we have

%We then define $\tilde k:=F[k]$ implicitly by mass splitting; that is,

\begin{equation}\label{eqn: tilde k definition}
	\nu_k(\underline y, y):=\mu(X_\geq( y, \tilde k(y))).
\end{equation}

%This defines a mapping $F: L^1\Big((\underline y, \overline y); [\underline d, \overline d]\Big) \rightarrow L^1\Big((\underline y, \overline y); [\underline d, \overline d]\Big)$.

\begin{lemma}\label{lemma: solutions are fixed points}
	Let $\nu$ be a minimizer of \eqref{eqn: variational problem} 	such that $(c,\mu,\nu)$ is nested.  Then 	$k(y) =v'(y)$	is a fixed point of $F$, where $v$ is the corresponding Kantorovich potential.
\end{lemma}

\begin{proof}
	The first order conditions corresponding to the variational problem are 
	$$
	v(y)+\log(\bar\nu(y)) =C
	$$
	for some constant $C$, $\nu$ almost everywhere (see, for example, Section 7.4 in \cite{santambook}), where  $\bar \nu$ is the density of minimizer $\bar \nu$ and $v$ the corresponding Kantorovich potential.
	This implies that that $\nu =\nu_k$.  
	
	Since $(c,\mu,\nu)$ is nested, the $\tilde k=F[K]$ defined by \eqref{eqn: tilde k definition} coincides with the derivative of the Kantorovich potential for optimal transport between $\mu$ and $\nu_k=\nu$, which is exactly $k$.
\end{proof}
We will prove that, under certain conditions, $F$ is a contraction.  The Banach fixed point theorem will then imply that $F$ has a unique fixed point.  Under conditions ensuring nestedness of the minimizer,  Lemma \ref{lemma: solutions are fixed points}, will then ensure that that fixed point is the minimizer of \eqref{eqn: variational problem}.

We need the following conditions on $X$, $c$ and $\mu$ to ensure $F$ is a contraction:
\begin{itemize}
	\item $A:=\max|D_xc_y(x,y)| <\infty$,
	\item $B:=\min\bar \mu(x)>0$,
	\item There is a $C >0$ such that, for all $y\in Y$,  $x\in X$, using $c_y(x,y)$ as a shorthand for $\frac{\partial c}{\partial y}(x,y)$,  
	\begin{equation}
	\mathcal H^{m-1}(X_1^c(y,c_y(x,y))) \geq C\min\big\{  \mu(X_{\geq}^c(y,c_y(x,y))), \mu(X_\leq^c(y,c_y(x,y)))\Big\} \label{cond: lower bound of level curves}.
	\end{equation}
\end{itemize}

	Let us discuss briefly the third condition above.  For any $x,y$, setting $p=c_y(x,y)$, the level set $X_1^c(y,p)$ is an $m-1$ dimensional submanifold, dividing the region $X$ into the sub and super level sets $X_{\geq}^c(y,p)$ and  $X_\leq^c(y,p)$.  The condition implies that if $X_1^c(y,p)$ is small, then at least one of  $\mu(X_{\geq}^c(y,p))$ or  $\mu(X_\leq^c(y,p))$ must also be (quantitatively) small.

In addition, we set $H(y) =\max\{\frac{|e^{(\overline{d} -2\underline{d})(y-\underline{y})}-1|}{|1-e^{-\overline{d}(y-\underline y)}|},\frac{ |e^{(\overline{d}-2\underline d)(\overline{y}-\underline{y})}-e^{(\overline{d}-2\underline d)(y-\underline{y}}|}{|(e^{-\overline{d}(y-\underline{y})}-e^{-\overline{d}(\overline{y}-\underline{y})}|} \}$.
\begin{theorem}\label{thm: contraction}
	Assume the three conditions above, and that
	$$
	\frac{2A\overline d^3(1-e^{ -\underline{d}(\overline y-\underline  y)})^2  }{BC\underline d^2 [1-e^{-\overline d(\overline y-\underline y)}]^2|\overline{d}-2\underline d|}\int_{\underline y}^{\overline y}H(y)dy <1.
	$$%$\frac{2A}{BC} (\overline y-\underline y )^2e^{(\overline d-\underline d) (\overline y-\underline y)+3(2M)}<1$.  
	Then $F$ is a contraction.
	
\end{theorem}

\begin{remark}
 The condition appearing in the statement of the theorem looks complicated; however, all quantities except for $B$ and $C$ can be computed using  the cost function alone ($B$ involves the reference marginal $\mu$  while $C$ involves both the cost and $\mu$).  Since the factor in front of the integral and the function $H(y)$ are bounded, the limit as $\overline{y} -\underline{y} \rightarrow 0$ of the left hand side above is $0$; therefore, the condition will always hold for sufficiently small intervals.  We illustrate these points in a simple example below, where we compute each of these quantites explicly. 
\end{remark}
\begin{example}

	Let $X = [-1,1]^2 \subset \mathbb{R}^2$, $Y=[0,\bar y] \subset \mathbb{R}$, and $c(x,y) =-x_1y$.  We take $\mu$ to be uniform measure, normalized to have total mass $1$, so that the density is $\bar \mu(x) =1/4$; clearly this means $B=1/4$.  We then have $c_y(x,y) =-x_1$   so that $\overline d =1$, $\underline d = -1$, and $D_xc_y(x,y) =-(1,0)$, so that $A=1$.  We also have that each $X_1^c(y,p) = \{x: x_1 =-p, \text{ }-1 <x_2<1\}$ is a vertical line and so $\mathcal{H}^{m-1}(X_1^c(y,c_y(x,y))) =\mathcal{H}^{1}(\{x: x_1 =-p, \text{ }-1 <x_2<1\}) =2$.  Since  $\mu(X_{\geq}^c(y,c_y(x,y)))+ \mu(X_\leq^c(y,c_y(x,y)) =1$, the minimum of the two values is at most $1/2$, so we can take $C=4$.

It is easy to check that $\frac{|e^{(\overline{d} -2\underline{d})(y-\underline{y})}-1|}{|1-e^{-\overline{d}(y-\underline y)}|} = \frac{e^{3y}-1}{1-e^{-y}}$ is increasing and so is bounded above by its value $\frac{e^{3\bar y}-1}{1-e^{-\bar y}}$ at the endpoint $\bar y$.  Similarly, $\frac{ |e^{(\overline{d}-2\underline d)(\overline{y}-\underline{y})}-e^{(\overline{d}-2\underline d)(y-\underline{y}}|}{|(e^{-\overline{d}(y-\underline{y})}-e^{-\overline{d}(\overline{y}-\underline{y})}|} = \frac{e^{3\bar y} -e^{3y}}{e^{-y}-e^{-\bar y}}$ is decreasing provided $\bar y$ is small enough, so it is bounded above by its value $ \frac{e^{3\bar y} -1}{1-e^{-\bar y}}$ at $y=0$.   Therefore, $H(y) \leq \frac{e^{3\bar y} -1}{1-e^{-\bar y}}$ and we have, for sufficiently small $\bar y$,
\begin{eqnarray*}
		\frac{2A\overline d^3(1-e^{ -\underline{d}(\overline y-\underline  y)})^2  }{BC\underline d^2 [1-e^{-\overline d(\overline y-\underline y)}]^2|\overline{d}-2\underline d|}\int_{\underline y}^{\overline y}H(y)dy &=& \frac{2 (1-e^{\bar y})^2}{\frac{1}{4} 4(1-e^{-\bar y})^23 }\int_0^{\bar y}H(y)dy \\
		& \leq & \frac{2 (1-e^{\bar y})^2}{3(1-e^{-\bar y})^2 }\bar y \frac{e^{3\bar y} -1}{1-e^{-\bar y}}
\end{eqnarray*}

This tends to $0$ as $\bar y$ does, since $\frac{ (1-e^{\bar y})^2}{(1-e^{-\bar y})^2 } \rightarrow  1$ and $\frac{e^{3\bar y} -1}{1-e^{-\bar y}} \rightarrow 3$, so that for sufficiently small $\bar y$, the condition in Theorem \ref{thm: contraction} holds.

\end{example}
\begin{remark}[Relationship to dual metrics]
	We show that $F$ is a contraction with respect to the $L^1$ metric.  Noting that the mapping $\nu \mapsto k$, where $k$ is defined by the mass splitting condition $\nu(\underline y, y) =\mu(X_\geq(y,k(y)))$ is a bijection on the set of non-atomic measures $\nu \in \mathcal P(Y)$ (and in particular on the set $\mathcal P^{ac}(Y)$ of measures which are absolutely continuous with respect to Lebesgue measure $\mathcal{L}$), we can consider the $L^1$ metric to be a metric on $\mathcal P^{ac}(Y)$.  If $(c,\mu,\nu)$ is nested, then $k(y)=v'(y)$ where $v$ is the Kantorovich dual potential.  In this case, this metric is exactly the dual metric  $W^*_{\mu,\mathcal L, c, 1}$ from Section \ref{sect:dual metrics}.  
\end{remark}
\begin{proof}

	Let $k_0,k_1 \in L^1\Big((\underline y, \overline y); [\underline d, \overline d]\Big)$. Then note that, setting  $v_i(y):=v_{k_i}(y)$,
	%$v_i(y):=v_{k_i}(y) \in [ d(y-\underline y),  -d(y-\underline y)]\subseteq [ d(\overline y-\underline y),  -d(\overline y-\underline y)]:=[-M,  M]$, and so 
		\begin{eqnarray}\label{eqn: difference of densities}
	|\bar \nu_0(y)-\bar \nu_1(y)|&=& \frac{e^{-v_0(y)}}{\int_{\underline y}^{\overline y}e^{-v_0(s)}ds}-\frac{e^{-v_1(y)}}{\int_{\underline y}^{\overline y}e^{-v_1(s)}ds}\nonumber\\
	&=&\Big |\frac{1}{\int_{\underline y}^{\overline y}e^{v_0(y)-v_0(s)}ds}-\frac{1}{\int_{\underline y}^{\overline y}e^{v_1(y)-v_1(s)}ds}\Big |\nonumber\\
	&=&\frac{|\int_{\underline y}^{\overline y}[e^{v_1(y)-v_1(s)}-e^{v_0(y)-v_0(s)}]ds|}{\int_{\underline y}^{\overline y}e^{v_0(y)-v_0(s)}ds\int_{\underline y}^{\overline y}e^{v_1(y)-v_1(s)}ds}\nonumber\\
	&\leq&\frac{\int_{\underline y}^{\overline y}|e^{v_1(y)-v_1(s)}-e^{v_0(y)-v_0(s)}|ds}{[e^{\underline d(y-\underline y)}\int_{\underline y}^{\overline y}e^{-\overline d(s-\underline y)}ds]^2}\nonumber\\
	&\leq&\frac{\overline d^2\int_{\underline y}^{\overline y}e^{\overline{d}(y-\bar y) -\underline{d}(s-\bar y)}[|v_0(y)-v_1(y)|+ |v_0(s)-v_1(s)|]ds}{[e^{\underline d(y-\underline y)}(1-e^{-\overline d(\overline y-\underline y)})]^2}\nonumber\\
	&\leq & \frac{2\overline d^2||v_0-v_1||_{L^\infty}e^{\overline{d}(y-\bar y) }\int_{\underline y}^{\overline y}e^{ -\underline{d}(s-\bar y)}ds}{[e^{\underline d(y-\underline y)}(1-e^{-\overline d(\overline y-\underline y)})]^2}\nonumber\\
	&\leq & \frac{2\overline d^2||v_0-v_1||_{L^\infty}e^{(\overline{d}-2\underline d)(y-\bar y) }(1-e^{ -\underline{d}(\overline y-\underline  y)})}{\underline d[1-e^{-\overline d(\overline y-\underline y)}]^2}\nonumber\\
	&\leq & \frac{2\overline d^2||k_0-k_1||_{L^1}e^{(\overline{d}-2\underline d)(y-\bar y) }(1-e^{ -\underline{d}(\overline y-\underline  y)})}{\underline d[1-e^{-\overline d(\overline y-\underline y)}]^2}.
	\end{eqnarray}
	%\begin{eqnarray}
	%\bar \nu_0(y)-\bar \nu_1(y)&=& \frac{e^{-v_0(y)}}{\int_{\underline y}^{\overline y}e^{-v_0(s)}ds}-\frac{e^{-v_1(y)}}{\int_{\underline y}^{\overline y}e^{-v_1(s)}ds}\\
	%&=&\frac{1}{\int_{\underline y}^{\overline y}e^{v_0(y)-v_0(s)}ds}-\frac{1}{\int_{\underline y}^{\overline y}e^{v_1(y)-v_1(s)}ds}\\
	%&=&\frac{\int_{\underline y}^{\overline y}[e^{v_1(y)-v_1(s)}-e^{v_0(y)-v_0(s)}]ds}{\int_{\underline y}^{\overline y}e^{v_0(y)-v_0(s)}ds\int_{\underline y}^{\overline y}e^{v_1(y)-v_1(s)}ds}\\
	%&\leq & \frac{(\overline y -\underline y)e^{2M}2||v_0-v_1||_{L^\infty}}{[e^{2M}(\overline y -\underline y)]^2}\\
	%&=&\frac{e^{3(2M)}2||v_0-v_1||_{L^\infty}}{(\overline y -\underline y)}\\
	%&\leq&e^{6M}2||k_0-k_1||_{L^1}
	%\end{eqnarray}
	Now, from the mass splitting definition of $\tilde k$, assuming without loss of generality that $\tilde k_0(y) \geq \tilde k_1(y)$, we compute, using the co-area formula
	\begin{eqnarray}\label{eqn: difference by pop. splitting}
	&&\int_{\underline y}^{y}[\bar \nu_1(s)-\bar \nu_0(s)]ds =\mu\{x:\tilde k_0(y) \geq c_y(x,y) \geq \tilde k_1(y)\}\\
	&&=\int_{\{x:\tilde k_0(y) \geq c_y(x,y) \geq \tilde k_1(y)\}}\bar \mu (x)dx\\
	&&= \int_{\tilde k_1(y)}^{\tilde k_0(y)}\big(\int_{X_=(y,k)}\frac{1}{|D_xc_y(x,y)|}\bar \mu(x)d\mathcal H^{m-1}(x)\big)dk\\
	&&\geq |\tilde k_0(y) -\tilde k_1(y)| \frac{1}{\max|D_xc_y(x,y)|}[\min\bar \mu(x)][\min_{\tilde k_1(y)\leq k\leq\tilde k_0(y)}\mathcal H^{m-1}(X_=(y,k))].\nonumber
	\end{eqnarray}
	Now, suppose  the minimum of $\mathcal H^{m-1}(X_=(y,k))$ over $k\in[\tilde k_1(y),\tilde k_0(y)]$ is attained at some $k$; by assumption \eqref{cond: lower bound of level curves}, either  
	\begin{equation}\label{eqn: sublevel lower bound}
	\mathcal H^{m-1}(X_=(y,k)) \geq C\mu(X_\leq (y,k))
	\end{equation} or  
	\begin{equation}\label{eqn: superlevel lower bound}
	\mathcal H^{m-1}(X_=(y,k)) \geq C\mu(X_\geq (y,k)).
	\end{equation}  We assume the second case (the proof in the first case is similar), in which case, we have
	$$
	\mathcal H^{m-1}(X_=(y,k)) \geq C\mu(X_\geq (y,k)) \geq C\mu(X_\geq (y,\tilde k_1(y))).
	$$
	Therefore, 
	\begin{eqnarray}
	\label{eqn: unitegrated inequality}
	&&|\tilde k_1(y) -\tilde k_0(y)|\leq \frac{2A\overline d^2||k_0-k_1||_{L^1}(1-e^{ -\underline{d}(\overline y-\underline  y)})  }{BC \underline d[1-e^{-\overline d(\overline y-\underline y)}]^2\mu[X_\geq(y, \tilde k_1(y))]}\int_{\underline y}^{y} e^{(\overline{d}-2\underline d)(s-\bar y)} ds\nonumber\\
	&&= \frac{2A\overline d^2||k_0-k_1||_{L^1}(1-e^{ -\underline{d}(\overline y-\underline  y)})  }{BC \underline d[1-e^{-\overline d(\overline y-\underline y)}]^2\mu[X_\geq(y, \tilde k_1(y))]}\frac{( e^{(\overline{d}-2\underline d)(y-\bar y)}-1) }{\overline{d}-2\underline d}\nonumber\\
	&&=\frac{2A\overline d^2||k_0-k_1||_{L^1}(1-e^{ -\underline{d}(\overline y-\underline  y)})  }{BC \underline d[1-e^{-\overline d(\overline y-\underline y)}]^2\nu_{ k_1}(\underline{y},y)}\frac{ (e^{(\overline{d}-2\underline d)(y-\bar y)}-1) }{\overline{d}-2\underline d}\nonumber\\
	&&\leq \frac{2A\overline d^3||k_0-k_1||_{L^1}(1-e^{ -\underline{d}(\overline y-\underline  y)})  }{BC\underline d^2 [1-e^{-\overline d(\overline y-\underline y)}]^2}\frac{( e^{(\overline{d}-2\underline d)(y-\bar y)}-1) }{\overline{d}-2\underline d}\frac{(1-e^{-\underline{d}(\overline{y}-\underline{y})})}{(1-e^{-\overline{d}(y-\underline{y})})}.\nonumber\\
	\end{eqnarray}
%\begin{eqnarray}
%\label{eqn: unitegrated inequality}
%	|\tilde k_1 -\tilde k_0|&\leq& \frac{A}{BC} e^{6M}2\int_{\underline y}^{y} ||k_0-k_1||_{L^1}\frac{1}{\mu[X_\geq(y, \tilde k_0)]}ds\\
%	&=&\frac{A}{BC} e^{6M}2 \int_{\underline y}^{y} ||k_0-k_1||_{L^1}\frac{1}{\nu_{k_0}(\underline y,y)]}ds\\
%	&\leq&\frac{A}{BC} e^{6M}2\frac{(\overline y-\underline y )e^{-\underline d (\overline y-\underline y)}}{e^{-\overline d (\overline y-\underline y)}}\int_{\underline y}^{y} ||k_0-k_1||_{L^1}\frac{1}{y-\underline y}ds\\
%	&=&\frac{2A}{BC} (\overline y-\underline y )e^{(\overline d-\underline d) (\overline y-\underline y)+6M}||k_0-k_1||_{L^1}.
%\end{eqnarray}
where we have used \eqref{eqn: lower bound on nu} in the last line.
If instead of \eqref{eqn: superlevel lower bound} we have \eqref{eqn: sublevel lower bound}, a very similar argument holds, in which $\mu[X_\geq(y, \tilde k_1(y))]$ and $\nu_{k_1}(\underline y,y)$ are replaced by  $\mu[X_\leq(y, \tilde k_0(y))]$ and $\nu_{k_0}(y,\overline y)$, the domain of integration in \eqref{eqn: difference by pop. splitting} is replaced by $[y,\overline y]$, and we use \eqref{eqn: lower bound on nu2} instead of \eqref{eqn: lower bound on nu}.  In this case, we obtain

$$
	|\tilde k_1(y) -\tilde k_0(y)| \leq \frac{2A\overline d^3||k_0-k_1||_{L^1}(1-e^{ -\underline{d}(\overline y-\underline  y)^2})  }{BC\underline d^2 [1-e^{-\overline d(\overline y-\underline y)}]^2}\frac{( e^{(\overline{d}-2\underline d)(\overline{y}-\underline{y})}-e^{(\overline{d}-2\underline d)(y-\underline{y})}) }{\overline{d}-2\underline d}\frac{1}{(e^{-\overline{d}(y-\underline{y})}-e^{-\overline{d}(\overline{y}-\underline{y})})}.
$$

 Therefore, for every $y \in (\underline{y},\overline{y})$, %\eqref{eqn: unitegrated inequality} 
we have
$$
|\tilde k_1(y) -\tilde k_0(y)| \leq \frac{2A\overline d^3||k_0-k_1||_{L^1}[1-e^{ -\underline{d}(\overline y-\underline  y)}]^2 }{BC\underline d^2 [1-e^{-\overline d(\overline y-\underline y)}]^2|\overline{d}-2\underline d|} H(y).
$$ Integrating yields:
	$$
	||\tilde k_1 -\tilde k_0||_{L^1} \leq\Big( \frac{2A\overline d^3[1-e^{ -\underline{d}(\overline y-\underline  y)}]^2  }{BC\underline d^2 |\overline{d}-2\underline d|[1-e^{-\overline d(\overline y-\underline y)}]^2}\int_{\underline y}^{\overline y}H(y)dy\Big)||k_0-k_1||_{L^1}
	$$
	which is the desired result.
\end{proof}
As discussed above, together with the  Lemma \ref{lemma: solutions are fixed points}, condition \eqref{eqn: suff condition for nestedness} ensuring nestedness of the solution, and the Banach Fixed Point Theorem, this yields convergence of a simple iterative algorithm to compute the solution.
\begin{corollary}
Under the assumptions in Theorem \ref{thm: contraction}, and as well as assumption \eqref{eqn: suff condition for nestedness}, for any $k \in L^1\Big((\underline y, \overline y); [\underline d, \overline d]\Big)$, the sequence $F^n[k]$ converges to a function $k_{fixed}$ whose anti-derivative is the Kantorovich potential corresponding to the unique minimizer $\nu$ of \eqref{eqn: variational problem}.
\end{corollary}

\clearpage
\appendix
\section{Hierarchical metrics with multiple reference measures}\label{sect: multiple reference measures}
The $\nu$-based Wasserstein metric can be extended to align with several different probability measures in a hierarchical way.  Given $\nu_1,\dots,\nu_k \in \mathcal P(X)$, we define iteratively $\Pi^1_{opt} =\Pi_{opt}(\nu_1,\mu_i)$ and $\Pi^j_{opt}(\nu_1,...,\nu_j, \mu_i) \subseteq \mathcal \mathcal \mathcal P(X^{j+1})$ by 
$$
\Pi^j_{opt}(\nu_1,...,\nu_j, \mu_i) =\argmin_{ \substack{\gamma_{y_1,...y_{j-1}x_i} \in \Pi^{j-1}_{opt}(\nu_1,..., \nu_{j-1},\mu_i) \\ \gamma_{y_j}=\nu_j}} \int_{X^{j+1}}|x_i-y_j|^2d\gamma(y_1,....,y_j,x_i).
$$
A natural analogue of $W_\nu$ is then defined by:
\begin{equation}\label{eqn: hierarchical metric}
W^2_{\nu_1,...\nu_k}(\mu_0,\mu_1)=\inf_{\gamma_{y_1...y_kx_i} \in \Pi^k_{opt}(\nu_1,...,\nu_k, \mu_i), i=0,1}\int_{X^{k+2}}|x_0-x_1|^2d\gamma(y_1,....,y_k,x_0,x_1).
\end{equation}
As before, optimal couplings can be recovered as limits of multi-marginal problems, where the weights on the interaction terms reflect the hierarchy of the measures $\nu_1,...,\nu_k$ in the definition of $W_{\nu_1,...\nu_k}$:
\begin{proposition}
	Let $\nu_1,....,\nu_k,\mu_{0},\mu_1$ be probability measures on $X$.  Consider the multi-marginal optimal transport problem:
	\begin{equation}\label{eqn: multi-marginal multi-reference}
	\int_{X^{k+2}}c_\epsilon(x_0,x_1,y_1,y_2,...y_k)d\gamma
	\end{equation}
	where $c_\epsilon =\epsilon^{n-1}|x_0-x_1|^2 +\epsilon^{n-2}(|x_0-y_{n-1}|^2 +|x_1-y_{n-1}|^2)+\epsilon^{n-3}(|x_0-y_{n-2}|^2 +|x_1-y_{n-2}|^2)+....+\epsilon(|x_0-y_2|^2 +|x_1-y_2|^2)+(|x_0-y_1|^2 +|x_1-y_1|^2)$.
	Then any weak limit $\bar \gamma$ of solutions $\gamma_\epsilon$ is optimal in \eqref{eqn: hierarchical metric}
\end{proposition}
\begin{proof}
	The proof is similar to the proof of the second part of Theorem \ref{thm: charaterization of metric}.  Letting $\gamma \in \Pi^k_{opt}(\nu_1,...,\nu_k, \mu_i)$ for $i=0,1$, optimality of $\gamma_\epsilon$ in \eqref{eqn: multi-marginal multi-reference} implies
\begin{equation}\label{eqn: optimality gamma epsilon multi-reference}
		\int_{X^{k+2}}c_\epsilon(x_0,x_1,y_1,y_2,...y_k)d \gamma_\epsilon \leq 	\int_{X^{k+2}}c_\epsilon(x_0,x_1,y_1,y_2,...y_k)d\gamma.
\end{equation}
	 Passing to the limit implies that $\bar \gamma \in \Pi_{opt}^1(\mu_i, \nu)$.  Now, combining the fact that $\gamma \in \Pi_{opt}^1(\mu_i, \nu)$ with \eqref{eqn: optimality gamma epsilon multi-reference} implies
\[\begin{split}
	 	&\int_{X^{k+2}}c_\epsilon(x_0,x_1,y_1,y_2,...y_k) - (|x_0-y_1|^2 +|x_1-y_1|^2)d\bar \gamma_\epsilon\\  &\leq 	\int_{X^{k+2}}c_\epsilon(x_0,x_1,y_1,y_2,...y_k)-(|x_0-y_1|^2 +|x_1-y_1|^2)d\gamma.
\end{split}
\]
	 Dividing by $\epsilon$ and taking the limit then implies that $\bar \gamma \in \Pi_{opt}(\nu_1,\nu_2, \mu_i)$ for $i=0,2$.  Continuing inductively in this way yields the desired result.
\end{proof}
As a consequence, we easily obtain the following result, which is similar to the main result in \cite{CarlierGalichonSantambrogio10}.
 \begin{corollary}
 	Let $k=n-1$, and $\nu_1,....\mu_{n-1}$ be probability measures concentrated on mutually orthogonal line segments, each absolutely continuous with respect to one dimensional Hausdorff measure.  Then solutions $\gamma_\epsilon$ to the multi-marginal optimal transport problem \eqref{eqn: multi-marginal multi-reference} converge weakly to $\gamma$, whose $1,2$ marginal is $(Id,G)_\#\mu_0$, where $G$ is the Knothe-Rosenblatt transport from $\mu_0$ to $\mu_1$.
 \end{corollary}

\bibliographystyle{plain}
\bibliography{bibli}

\end{document}